\documentclass[smallcondensed]{svjour3}      
\smartqed  
\usepackage{graphicx}

\journalname{Math. Program. Ser. A}

\usepackage{amssymb,amsmath, color}
\usepackage{amssymb}
\usepackage{amsbsy}
\usepackage{amsmath}
\usepackage{enumitem,algorithm2e,algorithmic}

\usepackage{ulem}

\usepackage[colorlinks=true]{hyperref}

\newcommand{\R}{{\mathbb R}}
\newcommand{\N}{{\mathbb N}}
\DeclareMathOperator{\argmin}{argmin}
\DeclareMathOperator{\prox}{prox}

\newcommand{\cC}{{\mathcal C}}

\newcommand{\cH}{{\mathcal H}}
\newcommand{\cO}{{\mathcal O}}

\newcommand{\Id}{\mathrm{I}}

\newcommand{\demi}{\frac{1}{2}}
\newcommand{\ie}{{\it i.e.}\,\,}

\newcommand{\Rb}{\R\cup\{+\infty\}}

\newcommand{\eqdef}{:=}
\newcommand{\seq}[1]{\pa{#1}_{k \in \N}}

\newcommand{\dotp}[2]{\langle #1,\,#2 \rangle}
\newcommand{\norm}[1]{\left\|{#1}\right\|}
\newcommand{\pa}[1]{\left({#1}\right)}
\newcommand{\brac}[1]{\left[{#1}\right]}
\newcommand{\IPAHD}{\mathrm{(IPAHD)}}
\newcommand{\IPAHDNS}{\text{\rm{(IPAHD-NS)}\,}}
\newcommand{\IPAHDSC}{\text{\rm{(IPAHD-SC)}\,}}
\newcommand{\IPAHDNSSC}{\text{\rm{(IPAHD-NS-SC)}\,}}
\newcommand{\IGAHD}{\text{\rm{(IGAHD)}\,}}
\newcommand{\IGAHDSC}{\text{\rm{(IGAHD-SC)}\,}}
\newcommand{\AVD}[1]{\text{\rm{(AVD)}$_{#1}$\,}}
\newcommand{\DINAVD}[1]{\text{\rm{(DIN-AVD)}$_{#1}$\,}}
\newcommand{\DIN}[1]{\text{\rm{(DIN)}$_{#1}$\,}}
\newcommand{\DINNS}[1]{\text{\rm{(DIN-NS)}$_{#1}$\,}}

\newcommand{\xkp}{x_{k+1}}

\newcommand{\tcb}[1]{{\color{black}{#1}}}

\usepackage{ulem}

\begin{document}

\title{First-order optimization algorithms via inertial systems with Hessian driven damping}
\titlerunning{Optimization via inertial systems with Hessian damping}

\author{Hedy Attouch \and Zaki Chbani \and Jalal Fadili
\and Hassan Riahi}
\authorrunning{H. Attouch, Z. Chbani, J. Fadili, H. Riahi} 

\institute{
H. Attouch \at IMAG, Univ. Montpellier, CNRS, Montpellier, France\\
\email{hedy.attouch@umontpellier.fr} \and 
Z. Chbani 
\at Cadi Ayyad Univ., Faculty of Sciences Semlalia, Mathematics, 40000 Marrakech, Morroco \\ \email{chbaniz@uca.ac.ma}
\and J. Fadili
\at{GREYC CNRS UMR 6072, Ecole Nationale Sup\'erieure d'Ing\'enieurs de Caen, France}\\
\email{Jalal.Fadili@greyc.ensicaen.fr} 
\and H. Riahi 
\at Cadi Ayyad Univ., Faculty of Sciences Semlalia, Mathematics, 40000 Marrakech, Morroco \email{h-riahi@uca.ac.ma}
}

\date{July 19, 2019}

\maketitle

\begin{abstract}
In a Hilbert space setting, for convex optimization, we analyze the convergence rate of a class of first-order algorithms involving inertial features. They can be interpreted as discrete time versions of inertial dynamics involving both viscous and Hessian-driven dampings. The geometrical damping driven by the Hessian intervenes in the dynamics in the form $\nabla^2  f (x(t)) \dot{x} (t)$. By treating this term as the time derivative of $ \nabla f (x (t)) $, this gives, in discretized form, first-order algorithms in time and space.
In addition to the convergence properties attached to Nesterov-type accelerated gradient methods, the algorithms thus obtained are new and show a rapid convergence towards zero of the gradients. On the basis of a regularization technique using the Moreau envelope, we extend these methods to non-smooth convex functions with extended real values. The introduction of time scale factors makes it possible to further accelerate these algorithms. We also report numerical results on structured problems to support our theoretical findings. 
\end{abstract}




\section{Introduction}\label{sec:prel} 

Unless specified, throughout the paper we make the following assumptions
\begin{equation}\tag{$\mathrm{H}$}\label{eq:mainassum}
\begin{cases}
\cH \text{ is a real Hilbert space};\\
f: \cH \rightarrow \R  \text{ is a convex function of  class $\mathcal C^2$}, \,  S := \argmin_{\cH} f \neq \emptyset;  \\
\gamma, \,  \beta,  \, b: [t_0, +\infty[ \to \R^+ \text{ are non-negative continuous functions}, \,  t_0 >0.
\end{cases}
\end{equation}

As a guide in our study, we will rely on the asymptotic behavior, when $t \to +\infty$, of the trajectories of  the inertial system with Hessian-driven damping
$$ 
\qquad \ddot{x}(t) + \gamma(t)\dot{x}(t) +  \beta (t) \nabla^2  f (x(t)) \dot{x} (t) + b(t)\nabla  f (x(t)) = 0 ,
$$
$ \gamma(t)$ and $\beta(t)$ are damping parameters, and $b(t)$ is a time scale parameter.

The time discretization of this system will provide a rich family of first-order methods for  minimizing  $f$.
At first glance, the presence of the Hessian may seem to entail numerical difficulties. However, this is not the case as the Hessian intervenes in the above ODE in the form $\nabla^2  f (x(t)) \dot{x} (t)$, which is nothing but the derivative w.r.t. time of $\nabla f (x(t))$. This explains why the time discretization of this dynamic provides first-order algorithms.
Thus, the Nesterov extrapolation scheme \cite{Nest1,Nest4} is modified by the introduction of the difference of the gradients at consecutive iterates. This gives algorithms of the form
\[
\begin{cases}
y_k = x_{k} + \alpha_k ( x_{k}  - x_{k-1}) - \beta_k  \left( \nabla f (x_k)  - \nabla f (x_{k-1}) \right) \\
x_{k+1} = T (y_k) ,
\end{cases}
\]
where $T$, to be specified later, is an operator involving the gradient or the proximal operator of $f$.

Coming back to the continuous dynamic, we will pay particular attention to the following two cases, specifically adapted to the properties of $ f$:
\begin{enumerate}[label=$\bullet$]
\item  For a general convex function $f$, taking
$ \gamma(t)=\frac{\alpha}{t} $,  gives
$$ {\DINAVD{\alpha, \beta, b}} \quad \ddot{x}(t) + \displaystyle{\frac{\alpha}{t} }\dot{x}(t) +  \beta (t) \nabla^2  f (x(t)) \dot{x} (t) + b(t)\nabla  f (x(t)) = 0.
$$
In the case $\beta \equiv 0$, $\alpha =3$,  $b(t)\equiv 1$,  it can be interpreted as a continuous version of the Nesterov accelerated gradient method \cite{SBC}. According to this, in this case, we will obtain $\cO\pa{t^{-2}}$ convergence rates for the objective values.

\item  For a $\mu$-strongly convex function $f$, we will rely on the autonomous inertial system with Hessian driven damping
\begin{equation*}
{\DIN{2\sqrt{\mu} , \beta}} \quad  \ddot{x}(t) + 2\sqrt{\mu} \dot{x}(t)   
 + \beta \nabla^2 f (x(t))\dot{x}(t) + \nabla f (x(t)) = 0,
\end{equation*}
and show exponential (linear) convergence rate for both objective values and gradients.
\end{enumerate}
For an appropriate setting of the parameters, the time discretization of these dynamics  provides first-order algorithms  with fast convergence properties. Notably, we will show a rapid convergence towards zero of  the gradients.

\subsection{A historical perspective}
B. Polyak initiated the use of inertial dynamics to accelerate the gradient method in optimization. In \cite{Pol,Polyak2}, based on the inertial system with a fixed viscous damping coefficient $\gamma >0$
 \begin{equation*}
{\rm (HBF)} \qquad \ddot{x}(t) + \gamma \dot{x}(t)  +  \nabla f (x(t)) = 0,
\end{equation*}
he introduced  the Heavy Ball with Friction method.   For a strongly convex function $f$, (HBF) provides convergence at exponential rate of  $f(x(t))$ to $\min_{\cH} f$. For general convex functions, the asymptotic convergence rate  of (HBF) is $\cO (\frac{1}{t})$ (in the worst case). This is however not better than the steepest descent. 
A decisive step to improve (HBF) was taken  by Alvarez-Attouch-Bolte-Redont \cite{AABR} by introducing the Hessian-driven damping term $\beta \nabla^2 f (x(t))\dot{x}(t)$, that is $\DIN{0,\beta}$. The next important step was accomplished by Su-Boyd-Cand\`es \cite{SBC} with the introduction of a vanishing viscous damping coefficient $\gamma(t)= \frac{\alpha}{t}$, that is $\AVD{\alpha}$ (see Section~\ref{AVD}).
The system $\DINAVD{\alpha, \beta, 1}$ (see Section~\ref{gen_conv_cont}) has emerged as a combination of $\DIN{0,\beta}$ and $\AVD{\alpha}$. Let us review some basic facts concerning these systems.
 
\subsubsection{The $\DIN{\gamma,\beta}$ dynamic} 
The inertial system 
\begin{equation*}
\DIN{\gamma,\beta} \qquad \ddot{x}(t) + \gamma \dot{x}(t) + \beta \nabla^2 f (x(t)) \dot{x}(t)  + \nabla f (x(t)) = 0,
\end{equation*}
was  introduced in \cite{AABR}. In line with (HBF), it contains a \textit{fixed} positive friction coefficient $\gamma$. The introduction of the Hessian-driven damping makes it possible to neutralize the transversal oscillations likely to occur with (HBF), as observed in \cite{AABR} in the case of the Rosenbrook function. The need to take a geometric damping adapted to $f$ had already been observed by Alvarez \cite{Alv0} who considered
\[
\ddot{x}(t) + \Gamma \dot{x}(t) + \nabla f (x(t)) = 0 ,
\] 
where $\Gamma: \cH \to \cH$ is a linear positive anisotropic operator. But still this damping operator is fixed. For a general convex function, the Hessian-driven damping in $\DIN{\gamma,\beta}$ performs a similar operation in a closed-loop adaptive way. The terminology (DIN) stands shortly for Dynamical Inertial Newton. It  refers to the natural link between this dynamic and the continuous Newton method. 

\subsubsection{The \AVD{\alpha} dynamic}\label{AVD}
 The inertial system  
 \begin{equation*}
\AVD{\alpha} \qquad \ddot{x}(t) + \frac{\alpha}{t} \dot{x}(t)  +  \nabla f (x(t)) = 0,
\end{equation*}
 was introduced in the context of convex optimization in \cite{SBC}. For  general convex functions it provides a continuous version of the accelerated gradient method of Nesterov. For $\alpha \geq 3$, each trajectory $x(\cdot)$ of \AVD{\alpha} satisfies the asymptotic rate of convergence of the values $f(x(t)) - \inf_{\cH}f =\cO \left(1 /t^2\right)$. As a specific feature, the viscous damping coefficient $\frac{\alpha}{t}$ vanishes (tends to zero) as time $t$ goes to infinity, hence the terminology. 
The convergence properties of the dynamic \AVD{\alpha} have been the subject of many recent studies, see \cite{AAD,AC1,AC2,AC2R-EECT,ACPR,ACR-subcrit,AP,AD,AD17,May,SBC}.
They helped to explain why $\frac{\alpha}{t}$ is a wise choise of the damping coefficient. 
 
In \cite{CEG1}, the authors showed that a vanishing damping coefficient $\gamma (\cdot)$ dissipates the energy, and hence makes the dynamic interesting for optimization, as long as $\int_{t_0}^{+\infty} \gamma (t) dt= + \infty$. The damping coefficient can go to zero asymptotically but not too fast. The smallest which is admissible is of order $\frac{1}{t}$. It enforces  the inertial effect with respect to the friction effect.
 
The tuning of the parameter $\alpha$ in front of $\frac{1}{t}$ comes from the Lyapunov analysis and the optimality of the convergence rates obtained. 
The case $ \alpha = 3 $, which corresponds to Nesterov's historical algorithm, is critical.  In the case $ \alpha = 3 $,  the question of the convergence of the trajectories remains an open problem (except in one dimension where convergence holds \cite{ACR-subcrit}).
 As a remarkable property, for $\alpha >3$, it has been shown by Attouch-Chbani-Peypouquet-Redont   
\cite{ACPR} that each trajectory converges weakly to a minimizer.  The corresponding algorithmic result has been obtained by Chambolle-Dossal \cite{CD}.
For $\alpha >3$, it is shown in \cite{AP} and \cite{May} that the asymptotic convergence rate of the values is actually  $o(1/t^2)$.
The subcritical case $\alpha \leq 3$ has  been examined by Apidopoulos-Aujol-Dossal\cite{AAD}  and Attouch-Chbani-Riahi \cite{ACR-subcrit}, with the convergence rate of the objective values $\displaystyle{\cO\pa{t^{-\frac{2\alpha}{3}}}}$.
These rates are optimal, that is, they can be reached, or approached arbitrarily close:

\noindent $\bullet$   $\alpha  \geq 3 $: the optimal rate $\displaystyle{\cO\pa{t^{-2}}}$ is achieved by taking  $f (x) = \|x\|^{r}$ with $r\to +\infty$ ($f$ become very flat around its minimum), see  \cite{ACPR}. 

\noindent $\bullet$  $\alpha  < 3 $: the optimal rate $\displaystyle{\cO\pa{t^{-\frac{2\alpha}{3}}}}$ is achieved by taking  $f(x) = \|x\|$,  see \cite{AAD}. 

\noindent The inertial system with  a general damping coefficient $\gamma(\cdot)$ was recently studied by Attouch-Cabot in \cite{AC1,AC2}, and Attouch-Cabot-Chbani-Riahi in \cite{AC2R-EECT}.

\subsubsection{The $\DINAVD{\alpha, \beta}$ dynamic}\label{DIN-AVD-continu}
The  inertial system 
$$\DINAVD{\alpha, \beta} \qquad \ddot{x}(t) + \displaystyle{\frac{\alpha}{t} }\dot{x}(t) +  \beta \nabla^2  f (x(t)) \dot{x} (t) + \nabla  f (x(t)) = 0,
$$
was introduced in \cite{APR}. It combines the two  types of  damping considered above. Its  formulation looks at a first glance more complicated than \AVD{\alpha}. In \cite{APR2}, Attouch-Peypouquet-Redont showed that
\DINAVD{\alpha, \beta} is equivalent to the first-order system in time and space
$$  \left\{
\begin{array}{l}
\dot x(t) +  \beta \nabla f (x(t))- \left( \frac{1}{\beta} -  \frac{\alpha}{t} \right)  x(t) +  \frac{1}{\beta} y(t) = 0;  	 \\
\rule{0pt}{10pt}
 \dot{y}(t)    -  \left( \frac{1}{\beta} -  \frac{\alpha}{t} +  \frac{\alpha \beta}{t^2}  \right) x(t)
 +\frac{1}{\beta} y(t) =0.
\end{array}\right.
$$
This provides a natural extension to $f: \cH \to \Rb$ proper lower semicontinuous and convex, just replacing the gradient by the subdifferential.

To get better insight, let us compare the two dynamics \AVD{\alpha} and \DINAVD{\alpha,\beta} on a simple quadratic minimization problem, in which case the trajectories can be computed in closed form as explained in Appendix~\ref{SS:solquad}. Take $\cH= \R^2$ and $ f(x_1,x_2)=\frac{1}{2}(x_1^2+1000x_2^2)$, which is ill-conditioned. We take parameters $\alpha=3.1$, $\beta=1$, so as to obey the condition $\alpha >3$. Starting with initial conditions: $(x_1(1),x_2(1))=(1,1)$, $(\dot x_1(1),\dot x_2(1))=(0,0)$, we have the trajectories displayed in Figure~\ref{fig2D}. This illustrates the typical situation of an ill-conditioned minimization problem, where the wild oscillations of \AVD{\alpha} are neutralized by the Hessian damping in \DINAVD{\alpha,\beta} (see Appendix~\ref{SS:solquad} for further details).
\begin{figure}
\begin{center}
\includegraphics[width=\textwidth]{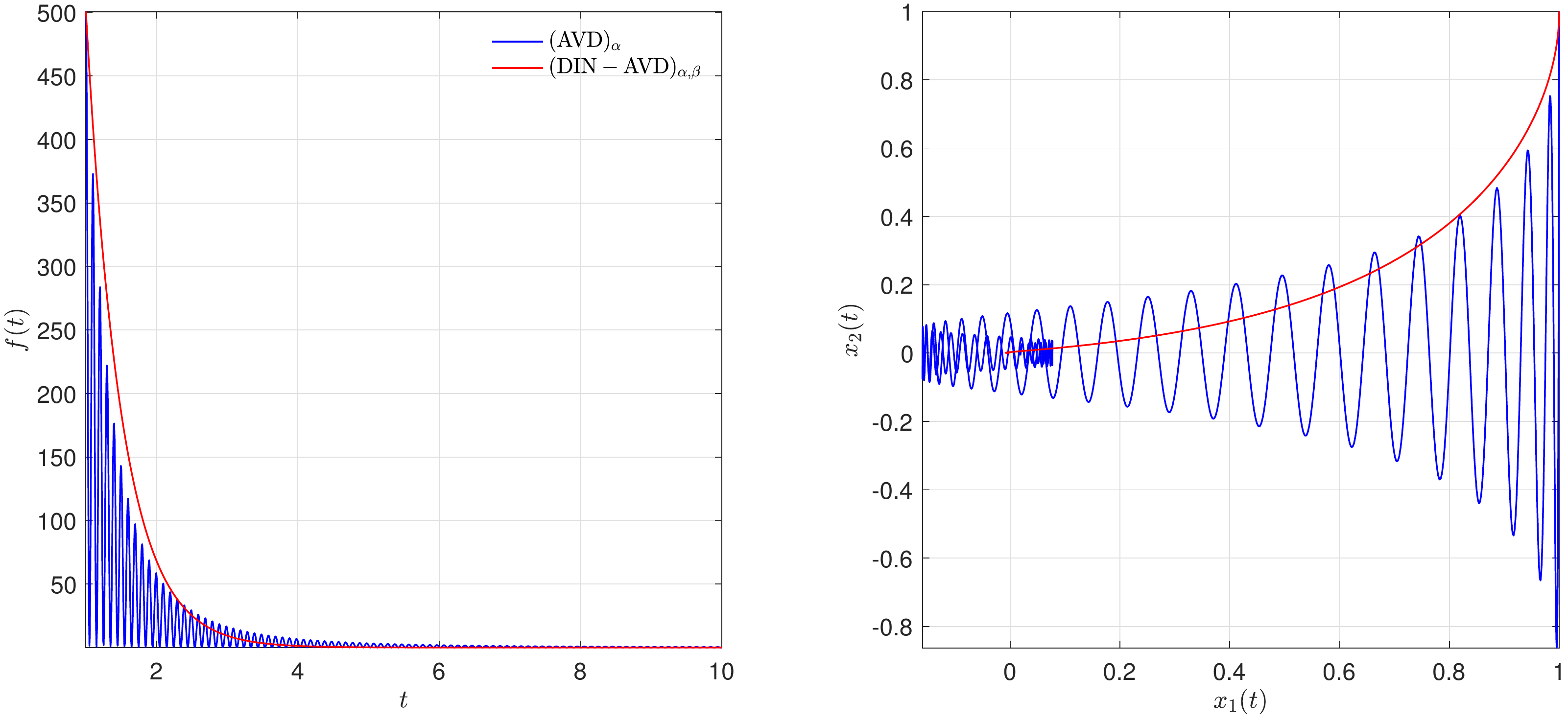}
\end{center}
\caption{Evolution of the objective (left) and trajectories (right) for \AVD{\alpha} ($\alpha=3.1)$ and \DINAVD{\alpha,\beta} ($\alpha=3.1,\beta=1$) on an ill-conditioned quadratic problem in $\R^2$.}
\label{fig2D}
\end{figure}

\subsection{Main algorithmic results}
Let us describe our main convergence rates for the gradient type algorithms. Corresponding results for the proximal algorithms are also obtained.

\paragraph{General convex function} Let  $f: \cH \to \R$ be  a  convex function whose gradient is $L$-Lipschitz continuous. Based on the discretization of $\DINAVD{\alpha, \beta, 1 + \frac{\beta}{t}}$, we consider
\[  
\begin{cases}
y_k=   x_{k} + \left(1- \frac{\alpha}{k}\right) ( x_{k}  - x_{k-1}) -
\beta \sqrt{s}   \left( \nabla f (x_k)  - \nabla f (x_{k-1}) \right) -  \frac{\beta \sqrt{s}}{k}\nabla f( x_{k-1})  	 \\
x_{k+1} = y_k - s \nabla f (y_k).
\end{cases}
\]
Suppose that $\alpha \geq 3$, $0 < \beta <  2 \sqrt{s}$, $sL \leq 1$. In Theorem~\ref{pr.decay_E_k}, we show that
\begin{eqnarray*}
&& (i) \, \,
f(x_k)-\min_{\cH} f = \cO 
\left(\dfrac{1}{k^2}\right)  \,  \text{ as } k\to +\infty;
\\
&& (ii)  \quad \sum_k  k^2 \| \nabla f (y_k) \|^2 < +\infty \text{ and } \sum_k  k^2 \| \nabla f (x_k) \|^2 < +\infty. \hspace{1cm}
\end{eqnarray*}

\paragraph{Strongly convex function} When $f: \cH \to \R$ is $\mu$-strongly convex for some $\mu >0$, our analysis relies on the autonomous dynamic ${\rm (DIN)_{\gamma, \beta}}$ with $\gamma = 2\sqrt{\mu}$. Based on its time discretization, we obtain linear convergence results for the values (hence the trajectory) and the gradients terms. Explicit discretization gives the inertial gradient algorithm

\begin{small}
\[
 x_{k+1} =   x_{k} + \frac{1 -\sqrt{\mu s}  }{1 +\sqrt{\mu s}}( x_{k}  - x_{k-1})-
\frac{\beta \sqrt{s} }{1 +\sqrt{\mu s}}  \left(\nabla f (x_k)  - \nabla f (x_{k-1})  \right) - \frac{s}{1 +\sqrt{\mu s}}\nabla f (x_k).
\]
\end{small}

\noindent Assuming that  $\nabla f$  is $L$-Lipschitz continuous, $ L$ sufficiently small and $  \beta \leq \displaystyle{ \frac{1}{\sqrt{\mu}}} $, it is shown in Theorem \ref{grad-inertiel-strongly convex} that, with
$q= \displaystyle{\frac{1}{1+ \demi \sqrt{\mu s}}}$ ( $0 <q <1$)
\[
f(x_k)-  \min_{\cH}f =\cO \pa{q^{k}} \quad \text{ and } \norm{x_k-x^\star} = \cO\pa{q^{k/2}} \mbox{  as} \,  k  \,\to + \infty ,
\]
Moreover, the gradients converge exponentially fast to zero.

\subsection{Contents}
The paper is organized as follows. Sections \ref{gen_conv_cont} and  
\ref{general_con_algo} deal with the case of general convex functions, respectively in the continuous case and the algorithmic cases. We improve the Nesterov convergence rates by showing in addition fast convergence of the gradients. Sections \ref{s_convex_cont} and \ref{s_convex-algo} deal with the same questions in the case of strongly convex functions, in which case, linear convergence results are obtained. Section~\ref{sec:numerics} is devoted to numerical illustrations. We conclude with some perspectives.

\section{Inertial dynamics for general convex functions}\label{gen_conv_cont}

Our analysis deals with the inertial system with Hessian-driven damping
\[
\DINAVD{\alpha, \beta, b} \quad \ddot{x}(t) + \frac{\alpha}{t} \dot{x}(t)  + \beta (t) \nabla^2 f (x(t))\dot{x} (t) +  b(t) \nabla f (x(t)) = 0.
\]
\subsection{Convergence rates}
\tcb{We start by stating a fairly general theorem on the convergence rates and integrability properties of $\DINAVD{\alpha,\beta,b}$ under appropriate conditions on the parameter functions $\beta(t)$ and $b(t)$. As we will discuss shortly, it turns out that for some specific choices of the parameters, one can recover most of the related results existing in the literature.}
The following quantities play a central role in our analysis: 
\begin{equation}\label{bacic-def}
w(t)\eqdef b(t)-\dot{\beta}(t) -\dfrac{\beta(t)}{t} \, \mbox{ and } \, \delta(t)\eqdef t^2 w(t).
\end{equation}
\begin{theorem} \label{ACFR,rescale}
\tcb{Consider $\DINAVD{\alpha,\beta,b}$, where \eqref{eq:mainassum} holds}. Take $\alpha \geq 1$.
Let  $x: [t_0, +\infty[ \rightarrow \cH$ be a solution trajectory of $\DINAVD{\alpha, \beta, b}$. Suppose that the following growth conditions are satisfied: 
\begin{eqnarray*}
&& (\mathcal{G}_{2}) \quad b(t) > \dot{\beta}(t) +\dfrac{\beta(t)}{t}; 
\\
&& (\mathcal{G}_{3}) \quad  t\dot{w}(t)\leq (\alpha-3)w(t).\hspace{3cm}
\end{eqnarray*}
Then, $w(t)$ is positive and 
\begin{eqnarray*}
&& (i) \, \,
f(x(t))-\min_{\cH} f = \cO 
\left(\dfrac{1}{t^2 w(t)}\right) \, \mbox{ as } \, t \to + \infty;
\\
&& (ii) \, \int_{t_{0}}^{+\infty} t^2 \beta(t) w(t)\norm{\nabla f(x(t))}^{2} dt<+\infty;
\\
&& (iii) \,  \int_{t_{0}}^{+\infty} t \Big(  (\alpha-3)w(t) - t\dot{w}(t)\Big)(f(x(t))-\min_{\cH}f) dt <+\infty .
\end{eqnarray*}
\end{theorem}
\begin{proof}
Given  $x^\star \in \argmin_{\cH} f$,  define for $t\geq t_0$
\begin{equation}\label{eq:lyapcont}
E(t)\eqdef\delta(t)(f(x(t))-f(x^\star))+ \frac{1}{2}\norm{v(t)}^{2},
\end{equation}
where \, $v(t)\eqdef(\alpha-1)(x(t)-x^\star)+t\pa{\dot{x}(t)+\beta(t)\nabla f(x(t)}.$

\noindent The function $E(\cdot)$ will serve as a Lyapunov function. Differentiating $E$ gives
\begin{equation}\label{der-E}
\dfrac{d}{dt}E(t)=\dot{\delta}(t)(f(x(t))-f(x^\star))+\delta(t) \dotp{\nabla f(x(t))}{\dot{x}(t)}+ \dotp{v(t)}{\dot{v}(t)}.
\end{equation}
Using equation $\DINAVD{\alpha, \beta, b}$, we have
 $$\begin{array}{lll}
\dot{v}(t) & = & \alpha\dot{x}(t)+\beta(t) \nabla f(x(t))
+ t\big[ \ddot{x}(t)+\dot{\beta}(t)\nabla f(x(t))+\beta(t) \nabla^{2}f(x(t))\dot{x}(t)\big]\vspace{2mm}\\ 
 & = & \alpha\dot{x}(t)+ \beta(t) \nabla f(x(t))+t\big[- \frac{\alpha}{t}\dot{x}(t)+(\dot{\beta}(t)-b(t))\nabla f(x(t))\big]\vspace{2mm}\\ 
 & = & t\big[\dot{\beta}(t)+\dfrac{\beta(t)}{t}-b(t)\big]\nabla f(x(t)).
\end{array} $$
Hence,
$$\begin{array}{lll}
\dotp{v(t)}{\dot{v}(t)} 
& =&   (\alpha-1)t\Big(\dot{\beta}(t)+\dfrac{\beta(t)}{t}-b(t)\Big)\dotp{\nabla f(x(t))}{x(t)-x^\star}\\
 & & +t^{2}\Big(\dot{\beta}(t)+\dfrac{\beta(t)}{t}-b(t)\Big)\dotp{\nabla f(x(t))}{\dot{x}(t)}\\ 
 &  & +t^{2}\beta(t)\Big(\dot{\beta}(t) +\dfrac{\beta(t)}{t}-b(t)\Big)\norm{\nabla f(x(t))}^{2}.
\end{array} $$
Let us go back to \eqref{der-E}. According to the choice of $\delta(t)$, the terms $\dotp{\nabla f(x(t))}{\dot{x}(t)}$ cancel, which gives
\[
\begin{array}{lll}
\dfrac{d}{dt}E(t)&=&\dot{\delta}(t)(f(x(t))-f(x^\star))+\frac{(\alpha-1)}{t}\delta(t)\dotp{\nabla f(x(t))}{x^\star-x(t)} \\
&-&\beta(t)\delta(t)\norm{\nabla f(x(t))}^{2}.
\end{array} 
\]
Condition $(\mathcal{G}_{2})$ gives $\delta(t) >0 $. Combining this equation with convexity of $f$,
\[ 
f(x^\star)-f(x(t)) \geq\dotp{\nabla f(x(t))}{x^\star-x(t)} ,
\]
we obtain the inequality
\begin{equation}\label{bacic-Liap-22}
\dfrac{d}{dt}E(t)+\beta(t)\delta(t)\norm{\nabla f(x(t))}^{2}+  \Big[\frac{(\alpha-1)}{t}\delta(t) -\dot{\delta}(t)\Big](f(x(t))-f(x^\star))\leq 0. 
\end{equation}
Then note that
\begin{equation}\label{bacic-Liap-22-b}
 \frac{(\alpha-1)}{t}\delta(t) -\dot{\delta}(t)= t \Big(  (\alpha-3)w(t) - t\dot{w}(t)\Big).
\end{equation}
Hence, condition $(\mathcal{G}_{3})$ writes equivalently
\begin{equation}\label{222}
 \frac{(\alpha-1)}{t}\delta(t) - \dot{\delta}(t) \geq 0,
\end{equation}
which, by \eqref{bacic-Liap-22}, gives $\dfrac{d}{dt}E(t)\leq 0$. Therefore, $E(\cdot)$ is non-increasing, and hence
$E(t) \leq E(t_0)$. Since all the terms that enter $E(\cdot)$ are nonnegative, \tcb{we obtain $(i)$}.
%
Then, by integrating \eqref{bacic-Liap-22} we get
\[
\int_{t_{0}}^{+\infty} \beta(t)\delta(t)\norm{\nabla f(x(t))}^{2} dt \leq E(t_0 )<+\infty ,
\]
and
\[
\int_{t_{0}}^{+\infty} t \Big(  (\alpha-3)w(t) - t\dot{w}(t)\Big)(f(x(t))-f(x^\star)) dt \leq E(t_0 )<+\infty ,
\]
which gives  $(ii)$ and $(iii)$, and completes the proof. \qed  
\end{proof}
%

\subsection{Particular cases}\label{subsec:particularcont}
\tcb{As anticipated above, by specializing the functions $\beta(t)$ and $b(t)$, we recover most known results in the literature; see hereafter for each specific case and related literature. For all these cases, we will argue also on the interest of our generalization.}

\paragraph{\textbf{Case~1}} The $\DINAVD{\alpha, \beta}$ system corresponds to $\beta(t) \equiv \beta$ and $b(t) \equiv 1$. In this case, $w(t)= 1- \frac{\beta}{t}$. Conditions 
$(\mathcal{G}_{2}) $ and $(\mathcal{G}_{3}) $ are satisfied by taking $\alpha > 3$ and $ t> \frac{\alpha-2}{\alpha -3} \beta$. Hence, as a consequence of Theorem \ref{ACFR,rescale}, we obtain the following result of Attouch-Peypouquet-Redont \cite{APR2}:
\begin{theorem}[\cite{APR2}] \label{APR,2016}
Let  $x: [t_0, +\infty[ \rightarrow \cH$ be a trajectory of the dynamical system $\DINAVD{\alpha, \beta}$. Suppose $\alpha > 3$.  Then
$$   f(x(t))-  \min_{\cH}f =\cO \left( \frac{1}{t^2} \right) \, \, \mbox{ and } \, \, \displaystyle{\int_{t_0}^{\infty} t^2 \|\nabla f (x(t)) \|^2  dt}    < + \infty.
$$
\end{theorem}

\paragraph{\textbf{Case~2}} The system$\DINAVD{\alpha, \beta, 1+ \frac{\beta}{t}}$, which corresponds to $\beta(t) \equiv \beta$ and $b(t) = 1+ \frac{\beta}{t}$, was considered in \cite{SDJS}. 
Compared to $\DINAVD{\alpha, \beta}$ it has the additional coefficient $\frac{\beta}{t}$ in front of the gradient term. This vanishing coefficient will facilitate the computational aspects while keeping the structure of the dynamic. Observe that in this case, $w(t)\equiv 1$. Conditions 
$(\mathcal{G}_{2}) $ and $(\mathcal{G}_{3}) $ boil down to  $\alpha \geq 3$. Hence, as a consequence of Theorem \ref{ACFR,rescale}, we obtain

\begin{theorem}  \label{APR,variant}
Let  $x: [t_0, +\infty[ \rightarrow \cH$ be a solution trajectory of the dynamical system $\DINAVD{\alpha, \beta, 1+ \frac{\beta}{t}}$. Suppose $\alpha\geq 3$. Then 
$$   f(x(t))-  \min_{\cH}f =\cO \left( \frac{1}{t^2} \right) \, \, \mbox{ and } \, \, \displaystyle{\int_{t_0}^{\infty} t^2 \|\nabla f (x(t)) \|^2  dt}    < + \infty.
$$
\end{theorem}

\paragraph{\textbf{Case~3}} The dynamical system \DINAVD{\alpha, 0, b}, which corresponds to $\beta(t) \equiv 0$, was considered by Attouch-Chbani-Riahi in \cite{ACR-rescale}. It comes also naturally from the time scaling of \AVD{\alpha}. In this case, we have $w (t) = b(t)$. Condition $(\mathcal{G}_{2})$ is equivalent to $b(t)> 0$.
$(\mathcal{G}_{3})$ becomes  
\[
t\dot{b}(t)\leq (\alpha-3) b(t),
\]
which is precisely the condition introduced in \cite[Theorem~8.1]{ACR-rescale}. Under this condition, we have the convergence rate 
\[
f(x(t))-\min_{\cH}f =\cO\left(\dfrac{1}{t^{2}b(t)}\right) \, \mbox{ as } \, t \to + \infty.
\]
This makes clear the acceleration effect due to the time scaling. For  $b(t)= t^r$, we have $ f(x(t))-\min_{\cH}f =\cO\pa{\dfrac{1}{t^{2+r}}}$, under the assumption $\alpha \geq 3+r$.

\paragraph{\textbf{Case~4}} Let us illustrate our results in the case $ b (t) = ct^b$, $ \beta (t) = t^\beta$. We have
$ w (t) = ct^b-(\beta +1)t^{\beta -1},  w'(t) =cbt^{b-1}-(\beta ^2-1)t^{\beta -2}. $
The conditions  $(\mathcal{G}_{2}),(\mathcal{G}_{3}) $ can be written respectively as:
\begin{equation}\label{condgi}
ct^b>(\beta +1)t^{\beta -1}\text{ and } c(b-\alpha+3)t^b\leq (\beta +1)(\beta -\alpha +2)t^{\beta -1}.
\end{equation}
When $b=\beta -1$, the conditions \eqref{condgi} are equivalent to
$
\beta<c-1 \; \text{ and } \; \beta\leq \alpha -2,
$
which gives the convergence rate $f(x(t))-\min_{\cH} f = \cO \left(\dfrac{1}{t^{\beta +1}}\right)$. 

Let us apply these choices to the quadratic function $f: (x_1,x_2)\in \R^2 \mapsto \left( x_1+x_2\right)^2/2$. $f$ is convex but not strongly so, and $\argmin_{\R^2}f =\{(x_1,x_2)\in\R^2 : x_2=-x_1\}$. The closed-form solution of the ODE with this choice of $\beta(t)$ and $b(t)$ is given in Appendix~\ref{SS:solquad}. We choose the values $\alpha=5, \beta=3, b=\beta-1=2$ and $c=5$ in order to satisfy condition \eqref{condgi}. The left panel of Figure~\ref{fig4} depicts the convergence profile of the function value, and its right panel the trajectories associated with the system $ \DINAVD{\alpha, \beta, b}$ for different scenarios of the parameters. Once again, the damping of oscillations due to the presence of the Hessian is observed.
\begin{figure}
\begin{center}
\includegraphics[width=\textwidth]{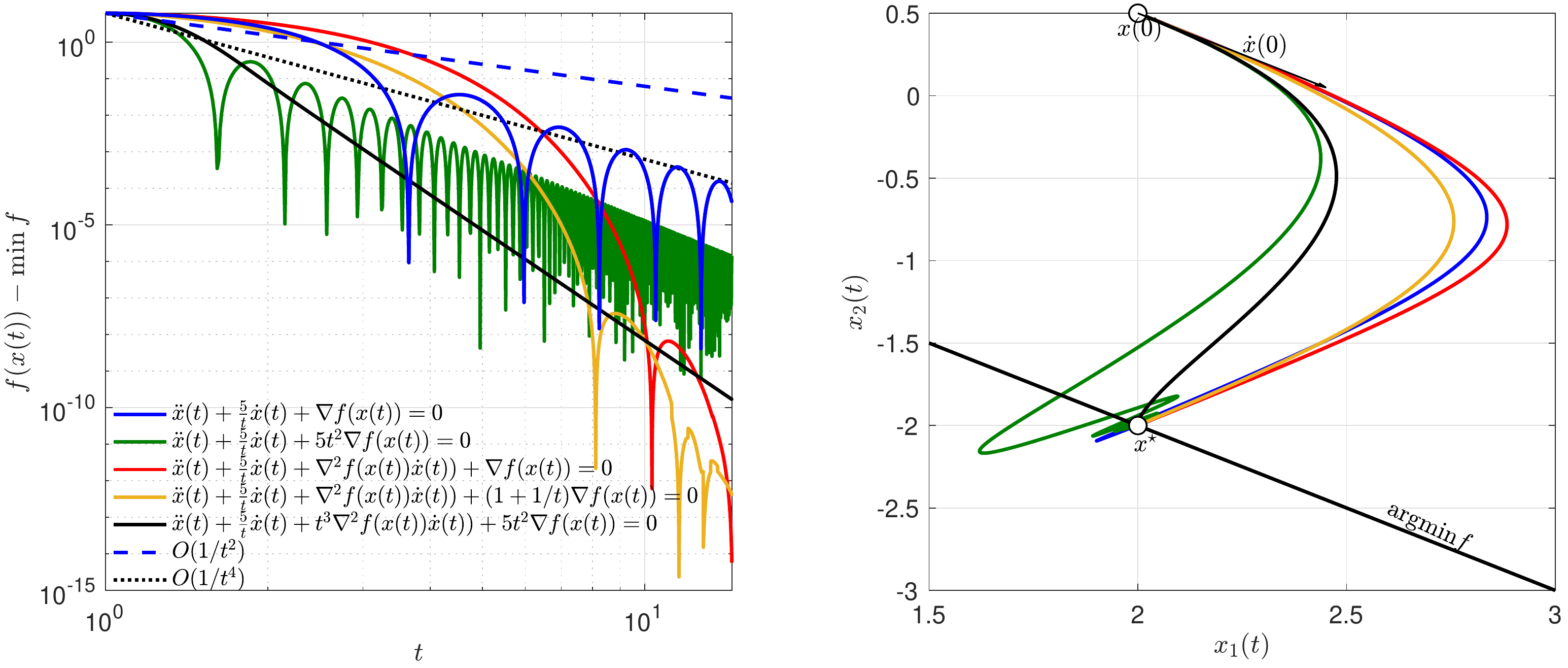}
\caption{Convergence of the objective values and trajectories associated with the system $ \DINAVD{\alpha, \beta, b}$ for different choices of $\beta(t)$ and $b(t)$.}
\label{fig4}
\end{center}
\end{figure}

\paragraph{\textbf{Discussion}}
\tcb{Let us first apply the above choices of $(\alpha,\beta(t),b(t))$ for each case to the quadratic function $f: (x_1,x_2)\in \R^2 \mapsto \left(x_1+x_2\right)^2/2$. $f$ is convex but not strongly so, and $\argmin_{\R^2}  f =\{(x_1,x_2)\in\R^2 : x_2=-x_1\}$. The closed-form solution of \DINAVD{\alpha, \beta, b} with each choice of $\beta(t)$ and $b(t)$ is given in Appendix~\ref{SS:solquad}. For all cases, we set $\alpha=5$. For case~1, we set $\beta=b=1$. For case~2, we take $\beta=1$. As for case~3, we set $r=2$. For case~4, we choose $\beta=3, b=\beta-1=2$ and $c=5$ in order to satisfy condition \eqref{condgi}. The left panel of Figure~\ref{fig4} depicts the convergence profile of the function value as well as the predicted convergence rates $\cO\pa{1/t^2}$ and $\cO\pa{1/t^4}$ (the latter is for cases with time (re)scaling). The right panel of Figure~\ref{fig4} displays the associated trajectories for the different scenarios of the parameters.

The rates one can achieve in our Theorem~\ref{ACFR,rescale} look similar to those in Theorem~\ref{APR,2016} and Theorem~\ref{APR,variant}. Thus one may wonder whether our framework allowing for more general variable parameters is necessary. The answer is affirmative for several reasons. First, our framework can be seen as a one-stop shop allowing for a unified analysis with an unprecedented level of generality. It also handles time (re)scaling straightforwardly by appropriately setting the functions $\beta(t)$ and $b(t)$ (see Case~3 and 4 above). In addition, though these convergence rates appear similar, one has to keep in mind that these are upper-bounds. It turns out from our detailed example in the quadratic case introduced above in Figure~\ref{fig4}, that not only the oscillations are reduced due to the presence of Hessian damping, but also the trajectory and the objective can be made much less oscillatory thanks to the flexible choice of the parameters allowed by our framework. This is yet again another evidence of the interest of our setting. 
}

\section{Inertial  algorithms for general convex functions}\label{general_con_algo}

\subsection{Proximal algorithms}
\subsubsection{Smooth case}
\tcb{Writing the term $\nabla^2f(x(t))\dot{x}(t)$ in $\DINAVD{\alpha, \beta, b}$ as the time derivative of $\nabla f(x(t))$, and taking the implicit time discretization of this system, with step size $h>0$, gives}
\begin{small}
\[
\frac{ x_{k+1} - 2x_{k}+ x_{k-1} }{h^2} +   \frac{\alpha}{kh} \frac{x_{k+1} - x_{k}}{h} + \frac{\beta_k}{h} (\nabla f( x_{k+1}) - \nabla f( x_{k})   ) +  b_k\nabla f( x_{k+1}) = 0.
\]
\end{small}
Equivalently
\begin{eqnarray}
&& k(x_{k+1} - 2x_{k}+ x_{k-1}) + \alpha(x_{k+1} - x_{k}) 
+ \beta_k hk (\nabla f( x_{k+1}) - \nabla f( x_{k})   ) \nonumber \\
&& + b_k h^2 k \nabla f( x_{k+1})=0. \label{basic-eq}
\end{eqnarray}
\tcb{Observe that this requires $f$ to be only of class $\mathcal C^1$}. Set now $s = h^2$. We obtain the following algorithm with $\beta_k$ and $b_k$ varying with $k$:
\begin{eqnarray*}
\boxed{
\begin{array}{rcl}
&& {\rm \mbox{$\IPAHD$: Inertial Proximal Algorithm with Hessian Damping.}}\\
\hline
\smallskip\\ 
&&  \mbox{Step} \, k:  \mbox{ Set} \,  \, \mu_k \eqdef \frac{k}{k + \alpha } (\beta_k \sqrt{s} + sb_k  ) .
\vspace{2mm}\\
&& {\IPAHD} \,  
\begin{cases}
y_k=   x_{k} + \left( 1- \frac{\alpha}{k + \alpha}\right) ( x_{k}  - x_{k-1}) + \beta_k \sqrt{s}   \left( 1- \frac{\alpha}{k + \alpha}\right) \nabla f (x_k)    \\
x_{k+1} =  \prox_{ \mu_k f}(y_k).
\end{cases}
\vspace{2mm}
\end{array}
}
\end{eqnarray*}
\begin{theorem} \label{ACFR,rescale-algo}
\tcb{Assume that $f: \cH \rightarrow \R$ is a convex $\mathcal C^1$ function. Suppose that $\alpha \geq 1$.}
Set
\begin{equation}\label{def:delta}
\delta_k \eqdef  h\Big(   b_k h k- \beta_{k+1} - k(\beta_{k+1} - \beta_{k})\Big)(k+1),
\end{equation}
 and suppose that the following growth conditions are satisfied: 
\begin{eqnarray*}
&& (\mathcal{G}^{\mathrm{dis}}_{2}) \quad b_k h k  -\beta_{k+1}  - k(\beta_{k+1} - \beta_{k}) > 0; \hspace{2cm}
\\
&& (\mathcal{G}^{\mathrm{dis}}_{3}) \quad \delta_{k+1} - \delta_k \leq (\alpha -1)
\frac{\delta_k}{k+1}.
\end{eqnarray*}
Then, $\delta_k$ is positive and, for any sequence $\seq{x_k}$ generated by $\IPAHD$
\begin{eqnarray*}
&& (i) \, \,
f(x_k)-\min_{\cH} f = \cO 
\left(\dfrac{1}{\delta_k}\right)= \cO 
\left(\dfrac{1}{k(k+1)\big( b_k h - \frac{\beta_{k+1}}{k} - (\beta_{k+1} - \beta_{k})\big)}\right) 
\\
&& (ii) \,  \sum_k \delta_k \beta_{k+1} \|\nabla f( x_{k+1})\|^2 < + \infty. 
\end{eqnarray*}
\end{theorem}

\tcb{
Before delving into the proof, the following remarks on the choice/growth of the parameters are in order. 
\begin{remark}
We first observe that condition $(\mathcal{G}^{\mathrm{dis}}_{2})$ is nothing but a forward (explicit) discretization of its continuous analogue $(\mathcal{G}_{2})$. In addition, in view of \eqref{bacic-def}, $(\mathcal{G}_{3})$ equivalently reads
\[
t\dot{\delta}(t)\leq (\alpha-1)\delta(t).
\]
In turn, \eqref{def:delta} and $(\mathcal{G}^{\mathrm{dis}}_{3})$ are explicit discretizations of \eqref{bacic-def} and $(\mathcal{G}_{3})$ respectively.
\end{remark}

\begin{remark}
The convergence rate on the objective values in Theorem~\ref{ACFR,rescale-algo}(i) is \linebreak$\cO\pa{1/((k+1)k}$ with the proviso that 
\begin{equation}\label{eq:G2disinf}
\inf_k (b_k h - \frac{\beta_{k+1}}{k} - (\beta_{k+1} - \beta_{k})) > 0,
\end{equation}
which in turn implies $(\mathcal{G}^{\mathrm{dis}}_{2})$. If, in addition to \eqref{eq:G2disinf}, we also have $\inf_k \beta_k > 0$, then the summability property in Theorem~\ref{ACFR,rescale-algo}(ii) reads $\sum_k k(k+1) \|\nabla f( x_{k+1})\|^2 < + \infty$. For instance, if $\beta_k$ is non-increasing and $b_k \geq c + \frac{\beta_{k+1}}{kh}$, $c > 0$, then \eqref{eq:G2disinf} is in force with $c$ as a lower-bound on the infimum. In summary, we get $\cO\pa{1/((k+1)k}$ under fairly general assumptions on the growth of the sequences $\seq{\beta_k}$ and $\seq{b_k}$.

Let us now exemplify choices of $\beta_k$ and $b_k$ that have the appropriate growth as above and comply with \eqref{eq:G2disinf} (hence $(\mathcal{G}^{\mathrm{dis}}_{2})$) as well as $(\mathcal{G}^{\mathrm{dis}}_{3})$. 
\begin{enumerate}[label=$\bullet$]
\item Let us take $\beta_k = \beta > 0$ and $b_k = 1$, which is the discrete analogue of the continuous case~1 considered in Section~\ref{subsec:particularcont} (recall that the continuous version was analyzed in \cite{APR2}). Note however that \cite{APR2} did not study the discrete (algorithmic) case and thus our result is new even for this system. In such a case, $\delta_k = h^2(k+1)(k-\beta/h)$ and $\beta_k$ is obviously non-icnreasing. Thus, if $\alpha > 3$, then one easily checks that \eqref{eq:G2disinf} (hence $(\mathcal{G}^{\mathrm{dis}}_{2})$) and $(\mathcal{G}^{\mathrm{dis}}_{3})$ are in force for all $k \geq \frac{\alpha-2}{\alpha-3}\frac{\beta}{h} + \frac{2}{\alpha-3}$.
\item Consider now the discrete counterpart of case~2 in Section~\ref{subsec:particularcont}. Take $\beta_k = \beta > 0$ and $b_k = 1+\beta/(hk)$\footnote{\tcb{One can even consider the more general case $b(t)=1+b/(hk), b > 0$ for which our discussion remains true under minor modifications. But we do not pursue this for the sake of simplicity.}}. Thus $\delta_k = h^2(k+1)k$. This case was studied in \cite{SDJS} both in the continuous setting and for the gradient algorithm, but not for the proximal algorithm. This choice is a special case of the one discussed above since $\beta_k$ is the constant sequence and $c=1$. Thus \eqref{eq:G2disinf} (hence $(\mathcal{G}^{\mathrm{dis}}_{2})$) holds. $(\mathcal{G}^{\mathrm{dis}}_{3})$ is also verified for all $k \geq \frac{2}{\alpha-3}$ as soon as $\alpha > 3$.
\end{enumerate}
\end{remark}
}

\begin{proof}
Given $x^\star  \in \argmin_{\cH} f$, set
\[
E_k\eqdef \delta_k ( f (x_k)-  f(x^\star) ) +\frac{1}{2}\norm{v_k}^2 ,
\]
where
\[
v_k \eqdef (\alpha -1) (x_k -x^\star) + k (x_{k} - x_{k-1} + \beta_k h \nabla f( x_{k})  ) ,
\]
and $\seq{\delta_k}$ is a positive sequence that will be adjusted. \tcb{Observe that $E_k$ is nothing but the discrete analogue of the Lyapunov function \eqref{eq:lyapcont}}. Set $\Delta E_k \eqdef E_{k+1}  - E_k $, i.e.,
\begin{small}
\begin{eqnarray*}
&& \Delta E_k = (\delta_{k+1} - \delta_k )( f (x_{k+1})- f(x^\star) )+ \delta_k  (f (x_{k+1}) -f (x_k)) 
+ \frac{1}{2}(\|v_{k+1}\|^2 -\|v_k\|^2)
 \end{eqnarray*}
\end{small}
Let us evaluate the last term of the above expression with the help of the three-point identity
$
\frac{1}{2}\norm{v_{k+1}}^2 -\frac{1}{2}\norm{v_k}^2  = \dotp{v_{k+1} -v_{k}}{v_{k+1}} - \frac{1}{2}\norm{v_{k+1} - v_{k}}^2  .
$

\noindent Using successively the definition of $v_k$ and \eqref{basic-eq}, we get
\begin{eqnarray*}
&&v_{k+1} - v_{k}= (\alpha -1) (x_{k+1} - x_{k}) +(k+1) (x_{k+1} - x_{k} + \beta_{k+1} h \nabla f( x_{k+1})  )\\
&&-k (x_{k} - x_{k-1} + \beta_k h \nabla f( x_{k}))\\
 &&= \alpha  (x_{k+1} - x_{k})   +k (x_{k+1} - 2x_{k} +x_{k-1} )+ \beta_{k+1} h \nabla f( x_{k+1}) \\
 &&+ h k(\beta_{k+1}\nabla f( x_{k+1})-\beta_{k}\nabla f( x_{k}) )    \\
 &&= [\alpha  (x_{k+1} - x_{k})   +k (x_{k+1} - 2x_{k} +x_{k-1} ) + kh \beta_{k}(\nabla f( x_{k+1})-\nabla f( x_{k}) ) ]
    \\
   &&+   \beta_{k+1} h \nabla f( x_{k+1})+ kh(\beta_{k+1} - \beta_{k})\nabla f( x_{k+1})\\
 &&= - b_k h^2 k \nabla f( x_{k+1}) +\beta_{k+1} h \nabla f( x_{k+1})+ kh(\beta_{k+1} - \beta_{k})\nabla f( x_{k+1})\\
 &&= h\Big( \beta_{k+1}  + k(\beta_{k+1} - \beta_{k}) - b_k h k \Big)\nabla f( x_{k+1}).
\end{eqnarray*}  
Set shortly $C_k=  \beta_{k+1}  + k(\beta_{k+1} - \beta_{k}) - b_k h k$. We have obtained
\begin{eqnarray*}
&&\frac{1}{2}\|v_{k+1}\|^2 -\frac{1}{2}\|v_k\|^2  = - \frac{h^2}{2}C_k^2 \|\nabla f( x_{k+1})\|^2  \\
&&
\dotp{\nabla f( x_{k+1})}{(\alpha -1) (x_{k+1} -x^\star)
 + (k+1) (x_{k+1} - x_{k} + \beta_{k+1} h \nabla f( x_{k+1}))}
\\
&&= - h^2\Big(\frac{1}{2} C_k^2 -C_k \beta_{k+1}\Big)\norm{\nabla f( x_{k+1})}^2 - (\alpha -1)h C_k\dotp{\nabla f( x_{k+1})}{x^\star- x_{k+1}} \\
&&- hC_k (k+1)\dotp{\nabla f( x_{k+1})}{x_k - x_{k+1}} .  
\end{eqnarray*} 
\tcb{By virtue of $(\mathcal{G}^\mathrm{dis}_{2})$, we have
$$
-C_k=  b_k h k  -\beta_{k+1}  - k(\beta_{k+1} - \beta_{k})  > 0.
$$
}
Then, in the above expression, the coefficient of $\|\nabla f( x_{k+1})\|^2$ is less or equal than zero, which gives
\begin{eqnarray*}
&&\frac{1}{2}\|v_{k+1}\|^2 -\frac{1}{2}\|v_k\|^2  \leq -(\alpha -1)hC_k\left\langle  \nabla f( x_{k+1}),x^\star- x_{k+1}  \right\rangle \\
&&-hC_k (k+1)\left\langle  \nabla f( x_{k+1}),x_k - x_{k+1} \right\rangle .  
\end{eqnarray*} 
According to the (convex) subdifferential inequality and \tcb{$C_k < 0$ (by $(\mathcal{G}^\mathrm{dis}_{2})$)}, we infer
\begin{eqnarray*}
&&\frac{1}{2}\|v_{k+1}\|^2 -\frac{1}{2}\|v_k\|^2  \leq -(\alpha -1)hC_k( f(x^\star) -f( x_{k+1})) \\
&&- hC_k (k+1) ( f(x_k) -f(x_{k+1})).
\end{eqnarray*} 
Take $\delta_k \eqdef - hC_k (k+1)=  h\Big(   b_k h k- \beta_{k+1} - k(\beta_{k+1} - \beta_{k})\Big)(k+1)$ so that the terms $f(x_k) - f(x_{k+1})$  cancel in $ E_{k+1}  - E_k$.
We obtain
$$ E_{k+1}  - E_k \leq \Big(\delta_{k+1} - \delta_k - (\alpha -1)h(  b_k h k - \beta_{k+1} - k(\beta_{k+1} - \beta_{k}))\Big) ( f(x_{k+1}) -f(x^\star))$$
Equivalently
$$ E_{k+1}  - E_k \leq \Big(\delta_{k+1} - \delta_k - (\alpha -1)
\frac{\delta_k}{k+1}\Big) ( f(x_{k+1}) -f(x^\star)).
$$
By assumption $(\mathcal{G}^\mathrm{dis}_{3})$,  we have  $\delta_{k+1} - \delta_k - (\alpha -1)
\frac{\delta_k}{k+1} \leq 0$.
Therefore, the sequence $\seq{E_k}$ is non-increasing, which, by definition of $E_k$, gives, for $k \geq 0$
\[
f(x_k)-\min_{\cH} f \leq \dfrac{E_0}{\delta_k}.
\]
By summing the inequalities 
$$
E_{k+1}  - E_k +  h\Big( \frac{h}{2} (\beta_{k+1}  + k(\beta_{k+1} - \beta_{k}) - b_k h k)^2  +\delta_k \beta_{k+1}\Big)\|\nabla f( x_{k+1})\|^2 \leq 0
$$
we finally obtain\, 
$
\sum_k \delta_k \beta_{k+1} \|\nabla f( x_{k+1})\|^2 < + \infty.\quad  
$ \qed
\end{proof}

\subsubsection{Non-smooth case}\label{non-smooth-prox}

Let $f: \cH \to \R \cup \left\lbrace +\infty \right\rbrace$ be a proper lower semicontinuous and convex function.
We rely on the basic properties of the Moreau-Yosida regularization. Let $f_{\lambda} $ be the Moreau envelope of $f$ of index $\lambda >0$, which is defined by: 
\[
f_{\lambda} (x) = \min_{z \in \cH} \left\lbrace f (z) + \frac{1}{2 \lambda} \norm{ z - x}^2   \right\rbrace, \quad \text{for any $x\in \cH$.} 
\]
We recall that  $f_{\lambda} $ is a convex function, whose gradient is $\lambda^{-1}$-Lipschitz continuous, such that $\argmin_{\cH} f_{\lambda} = \argmin_{\cH} f$.
The interested reader may refer to \cite{BC,Bre1} for a comprehensive treatment of the Moreau envelope in a Hilbert setting. Since the set of minimizers is preserved by taking the Moreau envelope, the idea is to replace $f$ by $f_{\lambda} $ in the previous algorithm, and take advantage of the fact that $f_{\lambda} $ is continuously differentiable. The Hessian dynamic attached to $f_{\lambda}$ becomes
\[
\quad \quad \ddot{x}(t) + \frac{\alpha}{t} \dot{x}(t) 
+ \beta \nabla^2 f_{\lambda} (x(t))\dot{x} (t) + b(t)\nabla f_{\lambda} (x(t)) = 0.
\]
However, we do not really need to work on this system (which requires $f_{\lambda}$ to be $\cC^2$), but with the discretized form which only requires the function to be continuously differentiable, as is the case of $f_{\lambda} $. \tcb{Then, algorithm $\IPAHD$ applied to $f_{\lambda}$ now reads}
\[
\left\{
\begin{array}{l}
y_k=   x_{k} + \left( 1- \frac{\alpha}{k + \alpha}\right) ( x_{k}  - x_{k-1})+
\beta \sqrt{s}   \left( 1- \frac{\alpha}{k + \alpha}\right) \nabla f_{\lambda} (x_k) \\
x_{k+1} =  \prox_{ \frac{k}{k + \alpha} (\beta \sqrt{s} + sb_k ) f_{\lambda}}(y_k). 
\end{array}\right.
\]
By applying Theorem~\ref{ACFR,rescale-algo} we obtain that under the assumption $(\mathcal{G}^\mathrm{dis}_{2})$ and $(\mathcal{G}^\mathrm{dis}_{3})$,
\begin{center}
$f_{\lambda} (x_k)-  \min_{\cH}f = \cO \left(\frac{1}{\delta_k}   \right) ,  \quad \sum_k \delta_k \beta_{k+1} \|\nabla f_{\lambda}( x_{k+1})\|^2 < + \infty.$
\end{center}
Thus, we just need to formulate these results in terms of $f$ and its proximal mapping. This is straightforward thanks to the following formulae from proximal calculus \cite{BC}:
\tcb{
\begin{align}
f_{\lambda} (x) &= f(\prox_{ \lambda f}(x)) + \frac{1}{2\lambda} \norm{x-\prox_{\lambda f}(x))}^2, \label{eq:env}\\
\nabla f_{\lambda} (x) &= \frac{1}{\lambda} \left( x-\prox_{ \lambda f}(x) \right), \label{eq:envgrad}\\
\prox_{ \theta f_{\lambda}}(x) &= \frac{\lambda}{\lambda +\theta}x + \frac{\theta}{\lambda +\theta}\prox_{ (\lambda + \theta)f}(x). \label{eq:envprox}
\end{align}
}

\noindent We obtain the following relaxed inertial proximal algorithm (NS stands for Non-Smooth):
\begin{eqnarray*}
\boxed{
\begin{array}{rcl}
&& \IPAHDNS:\\
\hline\\  
&& \, \mbox{Set} \,  \, \mu_k \eqdef \frac{\lambda(k + \alpha)}{\lambda(k + \alpha) + k (\beta \sqrt{s} + sb_k  ) }
\vspace{2mm}\\
 && 
\begin{cases}
y_k=   x_{k} + ( 1- \frac{\alpha}{k + \alpha}) ( x_{k}  - x_{k-1})+
\frac{\beta \sqrt{s}}{\lambda}    \left( 1- \frac{\alpha}{k + \alpha}\right) \left( x_k-\prox_{ \lambda f} (x_k)\right)    \\
x_{k+1} =  \mu_k y_k + (1-\mu_k)\prox_{ \frac{\lambda}{\mu_k} f} (y_k).
\end{cases}
\end{array}
}
\vspace{2mm}
\end{eqnarray*}
\tcb{
\begin{theorem} 
Let $f: \cH \to \R \cup \left\lbrace +\infty \right\rbrace$ be a convex, lower semicontinuous, proper function.
Let the sequence $\seq{\delta_k}$ as defined in \eqref{def:delta}, and suppose that the growth conditions $(\mathcal{G}^\mathrm{dis}_{2})$ and $(\mathcal{G}^\mathrm{dis}_{3})$ in Theorem~\ref{ACFR,rescale-algo} are satisfied.
Then, for any sequence $\seq{x_k}$ generated by \IPAHDNS, the following holds
\[
f(\prox_{ \lambda f}(x_k)) - \min_{\cH}f = \cO\pa{\frac{1}{\delta_k}}, \,
\sum_k \delta_k \beta_{k+1} \norm{x_{k+1} - \prox_{\lambda f}(x_{k+1})}^2 < + \infty .
\]
\end{theorem}}

\subsection{Gradient algorithms}\label{sec:igahd}
Take $f$ a convex function whose gradient is $L$-Lipschitz continuous. Our analysis is based on the dynamic  $\DINAVD{\alpha, \beta, 1+ \frac{\beta}{t}}$ considered in Theorem \ref{APR,variant}
with damping parameters $\alpha \geq 3$, $\beta \geq 0$.
Consider the time discretization of $\DINAVD{\alpha, \beta, 1+ \frac{\beta}{t}}$
\begin{eqnarray*}
&&\frac{1}{s}(x_{k+1} - 2x_{k}+ x_{k-1} ) +    \frac{\alpha}{ks}(x_{k} - x_{k-1}) +  \frac{\beta}{\sqrt{s}}(\nabla f( x_{k}) - \nabla f( x_{k-1})   ) \\
&&+ \frac{\beta}{k\sqrt{s}} \nabla f( x_{k-1})  +  
\nabla f( y_{k}) =0,
\end{eqnarray*}
with $y_k $ inspired by Nesterov's accelerated scheme. We obtain the following scheme:
\begin{eqnarray*}
\boxed{
\begin{array}{rcl}
&& \IGAHD: \text{Inertial Gradient Algorithm with Hessian Damping.} \\
\hline \\
\vspace{1mm}
&&
\text{Step $k$:} \, \alpha_k =  1- \frac{\alpha}{k}.\\
&& 
\begin{cases}
y_k=   x_{k} + \alpha_k ( x_{k}  - x_{k-1}) -
\beta \sqrt{s}   \left( \nabla f (x_k)  - \nabla f (x_{k-1}) \right) -  \frac{\beta \sqrt{s}}{k}\nabla f( x_{k-1})   
\\
x_{k+1} = y_k - s \nabla f (y_k) ,
\end{cases}
\vspace{1mm}
\end{array}
}
\end{eqnarray*}
Following \cite{AC2}, set $t_{k+1} =\frac{k}{\alpha-1}$, whence $t_k = 1 + t_{k+1} \alpha_k$.

\smallskip

\noindent Given  $x^\star \in \argmin_{\cH} f$, our Lyapunov analysis is based on the  sequence $\seq{E_k}$ 
\begin{eqnarray}
 && E_k\eqdef  t_k^2( f (x_k)-  f( x^\star) ) +\frac{1}{2s}\|v_k\|^2 \label{Lyap-function-grad1}\\
&& v_k \eqdef  (x_{k-1}  - x^\star) + t_k\Big(x_{k} - x_{k-1} + \beta \sqrt{s} \nabla f( x_{k-1})  \Big).\label{Lyap-function-grad2}
\end{eqnarray}

\begin{theorem}\label{pr.decay_E_k}
Let $f: \cH \to \R$ be a convex function whose gradient is $L$-Lipschitz continuous. Let $\seq{x_k}$ be a sequence generated by algorithm {\IGAHD}, where $\alpha \geq 3$, $0 \leq \beta <  2 \sqrt{s}$ and $s \leq 1/L$. Then the  sequence  $\seq{E_k}$ defined by {\rm\eqref{Lyap-function-grad1}-\eqref{Lyap-function-grad2}} is non-increasing, and the following convergence rates are satisfied:
\begin{eqnarray*}
&& (i) \, \,
f(x_k)-\min_{\cH} f = \cO 
\left(\dfrac{1}{k^2}\right)  \,  \text{ as } k\to +\infty;
\\
&& (ii) \text{ Suppose that }\,  \beta >0.  \, \text{ Then} \\
&&\vspace{2mm} \quad \sum_k  k^2 \| \nabla f (y_k) \|^2 < +\infty \text{ and } \sum_k  k^2 \| \nabla f (x_k) \|^2 < +\infty. \hspace{1cm}
\end{eqnarray*}
\end{theorem}
\begin{proof}
We rely on the following reinforced version of the gradient descent lemma (Lemma~\ref{ext_descent_lemma} in Appendix~\ref{SS:extended descent lemma}). Since $s \leq \frac{1}{L}$, and $\nabla f$ is $L$-Lipschitz continuous,
\begin{equation*}
f(y - s \nabla f (y)) \leq f (x) + \left\langle  \nabla f (y), y-x \right\rangle -\frac{s}{2} \|  \nabla f (y) \|^2 -\frac{s}{2} \| \nabla f (x)- \nabla f (y) \|^2
\end{equation*}
for all $x, y\in \cH$. Let us write it successively at $y=y_k$ and  $x= x_k$, then at $y=y_k$,  $x= x^\star$. According to $\xkp=y_k-s\nabla f (y_k)$ and $\nabla f (x^\star)=0$, we get
\begin{small}
\begin{align}
&f(x_{k+1}) \leq f(x_k) + \left\langle   \nabla f(y_k), y_k-x_k \right\rangle -\frac{s}{2} \|   \nabla f (y_k) \|^2 -\frac{s}{2} \| \nabla f (x_k)- \nabla f (y_k) \|^2 \label{eq.majo_Theta(x_k+1)-Theta(x_k)}\\
&f(x_{k+1}) \leq f(x^\star) + \left\langle   \nabla f (y_k), y_k-x^\star \right\rangle -\frac{s}{2} \|  \nabla f (y_k) \|^2 -\frac{s}{2} \| \nabla f (y_k) \|^2. \label{eq.majo_Theta(x_k+1)-min_Theta}
\end{align}
\end{small}
Multiplying \eqref{eq.majo_Theta(x_k+1)-Theta(x_k)} by $t_{k+1}-1\geq 0$, then adding \eqref{eq.majo_Theta(x_k+1)-min_Theta}, we derive that
\begin{eqnarray}\label{eq.combinaison_eq_prec_2}
&&t_{k+1}(f(x_{k+1}) -f(x^\star))\leq (t_{k+1}-1)(f(x_k)-f(x^\star)) \nonumber\\
&&+\langle  \nabla f (y_k) ,(t_{k+1}-1)(y_k-x_k)+y_k-x^\star\rangle -\frac{s}{2}t_{k+1}\| \nabla f (y_k) \|^2.\nonumber\\
&& - \frac{s}{2}(t_{k+1}-1) \| \nabla f (x_k)- \nabla f (y_k) \|^2 -
\frac{s}{2} \| \nabla f (y_k) \|^2.
\end{eqnarray}
Let us multiply \eqref{eq.combinaison_eq_prec_2} by $t_{k+1}$ to make appear $E_k$. We obtain 
\begin{eqnarray*}
&&t_{k+1}^2(f(x_{k+1}) -f(x^\star))\leq (t_{k+1}^2-t_{k+1}-t_k^2)(f(x_k)-f(x^\star)) + t_k^2(f(x_k)-f(x^\star)) \nonumber\\
&&+t_{k+1} \langle  \nabla f (y_k) ,(t_{k+1}-1)(y_k-x_k)+y_k-x^\star\rangle -\frac{s}{2}t_{k+1}^2\| \nabla f (y_k) \|^2\nonumber\\
&&- \frac{s}{2}(t_{k+1}^2-t_{k+1}) \| \nabla f (x_k)- \nabla f (y_k) \|^2 -
\frac{s}{2}t_{k+1} \| \nabla f (y_k) \|^2.
\end{eqnarray*}
Since $\alpha \geq 3$ we have $t_{k+1}^2-t_{k+1}-t_k^2 \leq 0$, which gives 
\begin{eqnarray*}
&&t_{k+1}^2(f(x_{k+1} -f(x^\star))\leq  t_k^2(f(x_k)-f(x^\star)) \nonumber\\
&&+t_{k+1} \langle  \nabla f (y_k) ,(t_{k+1}-1)(y_k-x_k)+y_k-x^\star\rangle -\frac{s}{2}t_{k+1}^2\| \nabla f (y_k) \|^2\nonumber \\
&&- \frac{s}{2}(t_{k+1}^2-t_{k+1}) \| \nabla f (x_k)- \nabla f (y_k) \|^2 -
\frac{s}{2}t_{k+1} \| \nabla f (y_k) \|^2.
\end{eqnarray*}
According to the definition of $E_k$, we infer
\begin{eqnarray*}
 E_{k+1}  - E_k &\leq& t_{k+1} \langle  \nabla f (y_k) ,(t_{k+1}-1)(y_k-x_k)+y_k-x^\star\rangle -\frac{s}{2}t_{k+1}^2\| \nabla f (y_k) \|^2\\
 && - \frac{s}{2}(t_{k+1}^2-t_{k+1}) \| \nabla f (x_k)- \nabla f (y_k) \|^2 -
\frac{s}{2}t_{k+1} \| \nabla f (y_k) \|^2 \\
&& + \frac{1}{2s}\|v_{k+1}\|^2 -\frac{1}{2s}\|v_k\|^2 .
 \end{eqnarray*}
Let us compute this last expression with the help of the elementary \tcb{identity}
$$
\frac{1}{2}\|v_{k+1}\|^2 -\frac{1}{2}\|v_k\|^2  = \left\langle  v_{k+1} -v_{k} , v_{k+1} \right\rangle - \frac{1}{2}\|v_{k+1} - v_{k}\|^2  .
$$
By definition of $v_k$, according to \IGAHD and  $t_k - 1= t_{k+1} \alpha_k$, we have
\begin{eqnarray*}
&&v_{k+1} - v_{k}= x_{k} - x_{k-1} + t_{k+1}(x_{k+1} - x_{k} + \beta \sqrt{s} \nabla f( x_{k})  )\\
&&-t_k (x_{k} - x_{k-1} + \beta \sqrt{s}  \nabla f( x_{k-1}))\\
&&= t_{k+1}(x_{k+1} - x_{k}) -(t_k -1)  (x_{k} - x_{k-1}) 
+\beta \sqrt{s} \Big( t_{k+1} \nabla f( x_{k})-  t_{k}\nabla f( x_{k-1}) \Big)  \\
&&= t_{k+1}\Big(x_{k+1} - (x_{k} +\alpha_k  (x_{k} - x_{k-1})\Big)   +\beta \sqrt{s} \Big( t_{k+1} \nabla f( x_{k})-  t_{k}\nabla f( x_{k-1}) \Big) \\
&&= t_{k+1}\left(x_{k+1} - y_k\right) -  t_{k+1}\beta \sqrt{s}
( \nabla f( x_{k})-  \nabla f( x_{k-1}) ) - t_{k+1} \frac{\beta \sqrt{s}}{k}\nabla f( x_{k-1})\\
&&+\beta \sqrt{s} ( t_{k+1} \nabla f( x_{k})-  t_{k}\nabla f( x_{k-1}) )\\
&&= t_{k+1}\left(x_{k+1} - y_k\right) +  \beta \sqrt{s} \left(t_{k+1} \pa{1- \frac{1}{k}} -   t_{k}  \right) \nabla f( x_{k-1}) \\
&&= t_{k+1}\left(x_{k+1} - y_k\right) = - st_{k+1} \nabla f( y_{k})  .
\end{eqnarray*}
Hence
\begin{eqnarray*}
&&\frac{1}{2s}\|v_{k+1}\|^2 -\frac{1}{2s}\|v_k\|^2  =  - \frac{s}{2}t_{k+1}^2\| \nabla f( y_{k}) \|^2 \\
&& - t_{k+1} \left\langle   \nabla f( y_{k}) , x_{k} - x^\star  +  t_{k+1}\Big(x_{k+1} - x_{k} + \beta \sqrt{s} \nabla f( x_{k})  \Big) \right\rangle .
\end{eqnarray*}
Collecting the above results, we obtain
\begin{eqnarray*}
 E_{k+1}  - E_k &\leq& t_{k+1} \langle  \nabla f (y_k) ,(t_{k+1}-1)(y_k-x_k)+y_k-x^\star\rangle -st_{k+1}^2\| \nabla f (y_k) \|^2\\
&&  - t_{k+1} \left\langle   \nabla f( y_{k}) , x_{k} - x^\star  +  t_{k+1}\Big(x_{k+1} - x_{k} + \beta \sqrt{s} \nabla f( x_{k}) \Big) \right\rangle \\
 && - \frac{s}{2}(t_{k+1}^2-t_{k+1}) \| \nabla f (x_k)- \nabla f (y_k) \|^2 -
\frac{s}{2}t_{k+1} \| \nabla f (y_k) \|^2 .
 \end{eqnarray*}
Equivalently
\begin{eqnarray*}
 E_{k+1}  - E_k &\leq& t_{k+1} \langle  \nabla f (y_k)  , A_k \rangle -st_{k+1}^2 \| \nabla f (y_k) \|^2 \\
  && - \frac{s}{2}(t_{k+1}^2-t_{k+1}) \| \nabla f (x_k)- \nabla f (y_k) \|^2 -
\frac{s}{2}t_{k+1} \| \nabla f (y_k) \|^2 ,
 \end{eqnarray*}
with
\begin{eqnarray*}
A_k &:=& (t_{k+1}-1)(y_k-x_k)+y_k  - x_{k}  -  t_{k+1}\Big(x_{k+1} - x_{k} + \beta \sqrt{s} \nabla f( x_{k})    \Big)\\
&=& t_{k+1}y_k - t_{k+1}x_k - t_{k+1}(x_{k+1} - x_{k})-  t_{k+1} \beta \sqrt{s} \nabla f( x_{k})\\
&=& t_{k+1}(y_k - x_{k+1} )-  t_{k+1} \beta \sqrt{s} \nabla f( x_{k})\\
&=& st_{k+1}\nabla f (y_k) -  t_{k+1} \beta \sqrt{s} \nabla f( x_{k})
 \end{eqnarray*}
Consequently
\begin{eqnarray*}
&& E_{k+1}  - E_k \leq t_{k+1} \dotp{\nabla f (y_k)}{st_{k+1}\nabla f (y_k) -  t_{k+1} \beta \sqrt{s} \nabla f( x_{k})}\\
 &&  -st_{k+1}^2 \| \nabla f (y_k) \|^2   - \frac{s}{2}(t_{k+1}^2-t_{k+1}) \| \nabla f (x_k)- \nabla f (y_k) \|^2 - \frac{s}{2}t_{k+1} \| \nabla f (y_k) \|^2 \\
&&=   - t_{k+1}^2 \beta \sqrt{s}
\dotp{\nabla f (y_k)}{\nabla f( x_{k})}  - \frac{s}{2}(t_{k+1}^2-t_{k+1}) \| \nabla f (x_k)- \nabla f (y_k) \|^2 \\
&& \qquad - \frac{s}{2}t_{k+1} \| \nabla f (y_k) \|^2 \\
&&=   - t_{k+1}B_k ,
\end{eqnarray*}
where
\[
B_k \eqdef t_{k+1} \beta \sqrt{s}\dotp{\nabla f (y_k)}{\nabla f( x_{k})}  + \frac{s}{2}(t_{k+1}-1) \| \nabla f (x_k)- \nabla f (y_k) \|^2 + \frac{s}{2} \| \nabla f (y_k) \|^2 .
\]
When $\beta=0$ we have $B_k\geq 0$. Let us analyze the  sign of $B_k$ in the case $\beta>0$. Set $Y=\nabla f (y_k)$, $X= \nabla f (x_k)$. We have
\begin{eqnarray*}
B_k & = &  \frac{s}{2} \| Y \|^2 + \frac{s}{2}(t_{k+1}-1) \| Y -X\|^2 
+ t_{k+1} \beta \sqrt{s}
\langle  Y  , X \rangle  \\
& = &   \frac{s}{2}t_{k+1} \| Y \|^2 + \left(t_{k+1}( \beta \sqrt{s} -s  ) +s  \right)
\langle  Y  , X \rangle  + \frac{s}{2}(t_{k+1}-1) \| X\|^2 \\
& \geq &  \frac{s}{2}t_{k+1} \| Y \|^2
 - \left(t_{k+1}( \beta \sqrt{s} -s  ) +s  \right)
\| Y\|  \| X \|  + \frac{s}{2}(t_{k+1}-1) \| X\|^2.
 \end{eqnarray*}
Elementary algebra gives that the above quadratic form is non-negative when
\[
\left(t_{k+1}( \beta \sqrt{s} -s  ) + s  \right)^2 \leq s^2 t_{k+1} (t_{k+1}-1).
\]
Recall that $t_k$ is of order $k$. Hence, this inequality is satisfied for $k$ large enough if  $( \beta \sqrt{s} -s  )^2 < s^2$, which is equivalent to $\beta < 2 \sqrt{s}.$
Under this condition
$
E_{k+1}  - E_k \leq 0$, which gives conclusion $(i)$. Similar argument gives that for $0 < \epsilon <  2 \sqrt{s}\beta - \beta^2$ (such $\epsilon$ exists according to assumption $0 < \beta < 2 \sqrt{s}$) 
\[
E_{k+1}  - E_k +  \demi \epsilon t_{k+1}^2 \| \nabla f (y_k) \|^2 \leq 0. 
\]
After summation of these inequalities, we obtain conclusion  $(ii)$. \qed
\end{proof}

\tcb{
\begin{remark}
In \cite[Theorem~8]{WRJ}, the same convergence rate as in Theorem~\ref{pr.decay_E_k} on the objective values is obtained for a very different discretization of the system $\DINAVD{\alpha, b\sqrt{s}, 1+ \frac{\alpha\sqrt{s}}{2t}}$. Their scheme is thus related but quite different from {\IGAHD}. Their claims require also intricate conditions relating $(\alpha,b,s,L)$ to hold true.

In Theorem~\ref{pr.decay_E_k}, the condition $\beta <  2 \sqrt{s}$ essentially reveals that in order to preserve acceleration offered by the viscous damping, the geometric damping should not be too large. It is an open question whether this constraint is a technical artifact or is fundamental to acceleration. We leave it to a future work.
\end{remark}
}

\begin{remark}
From $\sum_k  k^2 \| \nabla f (x_k) \|^2 < +\infty$ we immediately infer that for $k\geq 1$
\[
\inf_{i=1,\cdots,k} \| \nabla f (x_i) \|^2   \sum_{i=1}^k  i^2     \leq
\sum_{i=1}^{k}  i^2 \| \nabla f (x_i) \|^2   \leq \sum_{i \in \N}  i^2 \| \nabla f (x_i) \|^2 < +\infty.
\]
A similar argument holds for $y_k$. Hence  
\[ \inf_{i=1,...,k} \| \nabla f (x_i) \|^2 =\cO \left(  
\frac{1}{k^3} \right), \qquad \inf_{i=1,...,k} \| \nabla f (y_i) \|^2 =\cO \left(  
\frac{1}{k^3} \right). 
\]
\end{remark}

\begin{remark}
In Theorem~\ref{pr.decay_E_k}, the convergence property of the values is expressed according to the sequence $\seq{x_k}$. It is natural  to know if a similar result is true for the sequence $\seq{y_k}$. This is an open question in the case of Nesterov's accelerated gradient method and the corresponding FISTA algorithm for structured minimization~\cite{Nest4,BT}. In the case of the Hessian-driven damping algorithms, we give a partial answer to this question. By the classical descent lemma, and the monotonicity of $ \nabla f$ we have 
\begin{eqnarray*}
f(y_k) &\leq & f(x_{k+1}) + \langle  y_k - x_{k+1} , \nabla f ( x_{k+1}) \rangle + \frac{L}{2} \| y_k - x_{k+1} \|^2 \\
&\leq & f(x_{k+1}) + \langle  y_k - x_{k+1} , \nabla f ( y_{k}) \rangle + \frac{L}{2} \| y_k - x_{k+1} \|^2 
\end{eqnarray*}
According to $ x_{k+1} = y_k - s \nabla f (y_k)$ we obtain
\begin{eqnarray*}
f(y_k) - \min_{\cH} f &\leq & f(x_{k+1})  - \min_{\cH} f  + s \|\nabla f ( y_{k}) \|^2 + \frac{s^2 L}{2} \| \nabla f ( y_{k}) \|^2 .
\end{eqnarray*}
From Theorem~\ref{pr.decay_E_k} we deduce that
\[
f(y_k) - \min_{\cH} f \leq \cO\pa{\frac{1}{k^2}} + \pa{s + \frac{s^2 L}{2}} \| \nabla f ( y_{k}) \|^2 = \cO\pa{\frac{1}{k^2}} + o\pa{\frac{1}{k^2}} .
\]
\end{remark}

\begin{remark}
When $f$ is a proper lower semicontinuous proper function, but not necessarily smooth, we follow the same reasoning as in Section~\ref{non-smooth-prox}. We consider minimizing the Moreau envelope $f_\lambda$ of $f$, whose gradient is $1/\lambda$-Lipschitz continuous, and then apply \IGAHD to $f_\lambda$. We omit the details for the sake of brevity. This observation will be very useful to solve even structured composite problems as we will describe in Section~\ref{sec:numerics}.
\end{remark}

\section{Inertial dynamics for strongly convex functions}\label{s_convex_cont}

\subsection{Smooth case}
Recall the classical definition of strong convexity:
\begin{definition}
A function  $f: \cH \to \mathbb R$ is said to be $\mu$-strongly convex for some $\mu >0$ if $f- \frac{\mu}{2}\| \cdot\|^2$ is convex.
\end{definition}

\tcb{For strongly convex functions, a suitable choice of $\gamma$ and $\beta$  in $\DIN{\gamma,\beta}$ provides exponential decay of the value function (hence of the trajectory), and of the gradients.}

\begin{theorem}\label{strong-conv-thm}
\tcb{Suppose that \eqref{eq:mainassum} holds where $f: \cH \to \mathbb R$ is in addition $\mu$-strongly convex for some $\mu >0$.} Let $x(\cdot): [t_0, + \infty[ \to \cH$ be a solution trajectory of 
\begin{equation}\label{dyn-sc}
\ddot{x}(t) + 2\sqrt{\mu} \dot{x}(t) + \beta \nabla^2 f (x(t))\dot{x}(t) + \nabla f (x(t)) = 0.
\end{equation}
Suppose that $ 0 \leq \beta \leq \frac{1}{2\sqrt{\mu}}$. Then, the following hold:
\begin{enumerate}[label=(\roman*)]
\item for all $t\geq t_0$ 
\[
\frac{\mu}{2}\norm{x(t)-x^\star}^2 \leq f(x(t))-  \min_{\cH}f  \leq  C e^{-\frac{\sqrt{\mu}}{2}(t-t_0)}
\]
where $ C \eqdef f(x(t_0))-  \min_{\cH}f  + \mu \tcb{\|x(t_0)-x^\star\|^2} + \| \dot{x}(t_0)+ \beta \nabla f (x(t_0)) \|^2 .$

\item There exists some constant $C_1>0$ such that, for all $t\geq t_0$  
\[
e^{- \sqrt{\mu}t} \int_{t_0}^t  e^{\sqrt{\mu}s}\|  \nabla f (x(s))\|^2 ds \leq C_1 e^{-\frac{\sqrt{\mu}}{2}t}.
\]
Moreover, $\int_{t_0}^{\infty} e^{\frac{\sqrt{\mu}}{2}t}\| \dot{x}(t)\|^2 dt < + \infty.$ 
\end{enumerate}

\noindent When $\beta=0$, we have
$f(x(t))-  \min_{\cH}f =\cO \left( e^{-\sqrt{\mu}t} \right) \, \mbox{  as t } \,\to + \infty.$
\end{theorem}

\tcb{
\begin{remark}
When $\beta=0$, Theorem~\ref{strong-conv-thm} recovers \cite[Theorem~2.2]{Siegel}. In the case $\beta > 0$,  a result on a related but different dynamical system can be found in \cite[Theorem~1]{WRJ} (their rate is also sligthtly worse than ours). Our gradient estimate is distinctly new in the literature.
\end{remark}
}

\begin{proof}
\begin{enumerate}[label=(\roman*)]
\item Let $x^\star$ be the unique minimizer of $f$.
Define $\mathcal E : [t_0, +\infty[ \to \R^+ $  by
$$
\mathcal E (t)\eqdef f(x(t))-  \min_{\cH}f  + \frac{1}{2} \| \sqrt{\mu} (x(t) -x^\star) + \dot{x}(t)+ \beta \nabla f (x(t)) \|^2.
$$
Set $v(t)= \sqrt{\mu} (x(t) -x^\star) + \dot{x}(t)+ \beta \nabla f (x(t)) $.
Derivation of $\mathcal E (\cdot) $ gives 
$$
\frac{d}{dt}\mathcal E (t)\eqdef\langle  \nabla f (x(t)),  \dot{x}(t)       \rangle  + \langle v(t), \sqrt{\mu}  \dot{x}(t)  +  \ddot{x}(t)  +  \beta \nabla^2 f (x(t))\dot{x}(t)       \rangle .
$$
Using \eqref{dyn-sc}, we get
$$
\frac{d}{dt}\mathcal E (t)=\langle  \nabla f (x(t)),  \dot{x}(t)       \rangle  + \langle v(t), -\sqrt{\mu}  \dot{x}(t)  - \nabla f (x(t))       \rangle .
$$
After developing and simplification, we obtain
\begin{eqnarray*}
&&\frac{d}{dt}\mathcal E (t) + \sqrt{\mu}\langle  \nabla f (x(t)),  x(t) -x^\star       \rangle  + \mu\langle x(t) -x^\star , \dot{x}(t)\rangle + \sqrt{\mu} \| \dot{x}(t) \|^2  \\
&&+ \beta \sqrt{\mu} \langle  \nabla f (x(t)),  \dot{x}(t)       \rangle + \beta \| \nabla f (x(t)) \|^2  = 0.
\end{eqnarray*}
By strong convexity of $f$ we have
$$
\langle  \nabla f (x(t)),  x(t) -x^\star       \rangle  \geq f(x(t))- f(x^\star) + \frac{\mu}{2} \| x(t) -x^\star \|^2 .
$$
Thus, combining the last two relations we obtain 
$$
\frac{d}{dt}\mathcal E (t) + \sqrt{\mu}A \leq 0,
$$
where (the variable $t$ is omitted to lighten the notation)
$$
A \eqdef f(x)- f(x^\star) + \frac{\mu}{2} \| x -x^\star \|^2  + \sqrt{\mu} \langle x -x^\star , \dot{x}\rangle +  \| \dot{x} \|^2  + \beta \langle  \nabla f (x),  \dot{x}       \rangle + \frac{\beta}{\sqrt{\mu}} \| \nabla f (x) \|^2  
$$
Let us formulate  $A$ with $\mathcal E (t)$.
\begin{eqnarray*}
&& A=\mathcal E     - \frac{1}{2} \|  \dot{x}+ \beta \nabla f (x) \|^2   -  \sqrt{\mu} \langle x -x^\star , \dot{x} + \beta \nabla f (x)\rangle        + \sqrt{\mu} \langle x -x^\star , \dot{x}\rangle +  \| \dot{x} \|^2  \\
&&+ \beta \langle  \nabla f (x),  \dot{x}       \rangle + \frac{\beta}{\sqrt{\mu}} \| \nabla f (x) \|^2 . 
\end{eqnarray*}
After developing and simplifying, we obtain
$$
\frac{d}{dt}\mathcal E (t) + \sqrt{\mu}\left(\mathcal E (t)      
+ \frac{1}{2} \|  \dot{x} \|^2  + \left(\frac{\beta}{\sqrt{\mu}}-\frac{\beta^2}{2}\right) \| \nabla f (x) \|^2      -  \beta\sqrt{\mu} \langle x -x^\star ,  \nabla f (x)\rangle        \right) \leq 0.
$$
Since $ 0 \leq \beta \leq \frac{1}{\sqrt{\mu}}$, we immediately get
$\frac{\beta}{\sqrt{\mu}}-\frac{\beta^2}{2} \geq \frac{\beta}{2\sqrt{\mu}} .$ Hence
$$
\frac{d}{dt}\mathcal E (t) + \sqrt{\mu}\left(\mathcal E (t)      
+ \frac{1}{2} \|  \dot{x} \|^2  + \frac{\beta}{2\sqrt{\mu}}  \| \nabla f (x) \|^2      -  \beta\sqrt{\mu} \langle x -x^\star ,  \nabla f (x)\rangle        \right) \leq 0.
$$
Let us use again the strong convexity of $f$ to write
$$
\mathcal E (t) = \frac{1}{2}\mathcal E (t) + \frac{1}{2}\mathcal E (t)  \geq \frac{1}{2}\mathcal E (t)+ \frac{1}{2} \left(f(x(t))- f(x^\star)\right) \geq 
\frac{1}{2}\mathcal E (t) +  \frac{\mu}{4} \| x(t) -x^\star \|^2 .
$$
By combining the two inequalities above, we obtain
$$
\frac{d}{dt}\mathcal E (t) + \frac{\sqrt{\mu}}{2}\mathcal E (t)       + \frac{\sqrt{\mu}}{2}\|  \dot{x} (t)\|^2  +  \sqrt{\mu} B  \leq 0,
$$
where $B= \frac{\mu}{4} \| x(t) -x^\star \|^2 + \frac{\beta}{2\sqrt{\mu}}  \| \nabla f (x) \|^2      -  \beta\sqrt{\mu} \| x -x^\star\|  \|\nabla f (x)  \| $.\\
Set $X=\| x -x^\star\|$, $Y=  \| \nabla f (x) \|$. Elementary algebraic computation gives that, under the condition $ 0 \leq \beta \leq \frac{1}{2\sqrt{\mu}}$
$$
\frac{\mu}{4} X^2 + \frac{\beta}{2\sqrt{\mu}} Y^2 -  \beta\sqrt{\mu}XY \geq 0.
$$
Hence for $ 0 \leq \beta \leq \frac{1}{2\sqrt{\mu}}$
$$
\frac{d}{dt}\mathcal E (t) + \frac{\sqrt{\mu}}{2}\mathcal E (t)       + \frac{\sqrt{\mu}}{2}\|  \dot{x} (t)\|^2  \leq 0.
$$
By integrating the differential inequality above we obtain
$$
\mathcal E (t)   \leq \mathcal E (t_0) e^{-\frac{\sqrt{\mu}}{2}(t-t_0)}  .
$$
By definition of $\mathcal E (t)$, we infer
$$
 f(x(t))-  \min_{\cH}f \leq \mathcal E (t_0) e^{-\frac{\sqrt{\mu}}{2}(t-t_0)} ,
$$
and 
$$ \| \sqrt{\mu} (x(t) -x^\star) + \dot{x}(t)+ \beta \nabla f (x(t)) \|^2 \leq 2 \mathcal  E (t_0) e^{-\frac{\sqrt{\mu}}{2}(t-t_0)} .
$$
\item Set $C= 2\mathcal E (t_0) e^{\frac{\sqrt{\mu}}{2}t_0}$. Developing the above expression, we obtain
\begin{eqnarray*}
&&\mu \| x(t) -x^\star\|^2 + \|  \dot{x}(t)\|^2 + \beta^2  \|  \nabla f (x(t))\|^2    + 2\beta \sqrt{\mu}\left\langle x(t) -x^\star, \nabla f (x(t))  \right\rangle \\
&& +   \left\langle \dot{x}(t),  2 \beta\nabla f (x(t)) +  2\sqrt{\mu}  ( x(t) -x^\star)\right\rangle  \leq C e^{-\frac{\sqrt{\mu}}{2}t}.
\end{eqnarray*}
By  convexity of $f$ we have $\left\langle x(t) -x^\star, \nabla f (x(t))  \right\rangle \geq f(x(t)) - f(x^\star)$. Moreover, 
\begin{eqnarray*}
&&\left\langle \dot{x}(t),  2 \beta\nabla f (x(t)) +  2\sqrt{\mu}  ( x(t) -x^\star)\right\rangle \\
&&= \frac{d}{dt} \left(2 \beta (f(x(t)) - f(x^\star)    )+ \sqrt{\mu} \| x(t) -x^\star\|^2  \right).  
\end{eqnarray*}
Combining the above results, we obtain
\begin{eqnarray*}
&&\sqrt{\mu} [ 2 \beta (f(x(t)) - f(x^\star)) +\sqrt{\mu}  \| x(t) -x^\star\|^2 ] + \beta^2  \|  \nabla f (x(t))\|^2     \\
&& +   \frac{d}{dt} \left(2 \beta (f(x(t)) - f(x^\star)    )+ \sqrt{\mu} \| x(t) -x^\star\|^2  \right) \leq C e^{-\frac{\sqrt{\mu}}{2}t}.
\end{eqnarray*}
Set $Z(t)\eqdef 2 \beta (f(x(t)) - f(x^\star)) +\sqrt{\mu}  \| x(t) -x^\star\|^2 ]$. We have
$$
\frac{d}{dt} Z(t) + \sqrt{\mu} Z(t) + \beta^2  \|  \nabla f (x(t))\|^2 
 \leq C e^{-\frac{\sqrt{\mu}}{2}t}.
$$
By integrating this differential inequality,  elementary computation gives 
$$
e^{- \sqrt{\mu}t} \int_{t_0}^t  e^{ \sqrt{\mu}s}\|  \nabla f (x(s))\|^2 ds
 \leq C e^{-\frac{\sqrt{\mu}}{2}t}.
$$
Noticing that the integral of $e^{ \sqrt{\mu}s}$ over $[t_0, t]$ is of order $e^{ \sqrt{\mu}t}$, the above estimate reflects the fact, as $t \to + \infty$,  the gradient terms $\|  \nabla f (x(t))\|^2$ tend to zero at exponential rate (in average, not pointwise). \qed
\end{enumerate}
\end{proof}

\begin{remark}{Let us justify the choice of $\gamma = 2\sqrt{\mu}$ in Theorem~\ref{strong-conv-thm}.
Indeed, considering 
$$
\ddot{x}(t) + 2\gamma \dot{x}(t)   
 + \beta \nabla^2 f (x(t)) + \nabla f (x(t)) = 0,
$$
a similar proof to that described above can be performed on the basis of the Lyapunov function
$$
\mathcal E (t)\eqdef f(x(t))-  \min_{\cH}f  + \frac{1}{2} \| \gamma (x(t) -x^\star) + \dot{x}(t)+ \beta \nabla f (x(t)) \|^2.
$$
Under the conditions 
$
\gamma \leq \sqrt{\mu}  \, \mbox{ and   }\,  \beta \leq \frac{\mu}{2 \gamma^3}
$
we obtain the exponential convergence rate
$$
 f(x(t))-  \min_{\cH}f =\displaystyle{\cO \left( e^{-\frac{\gamma}{2}t} \right) }\, \mbox{  as } \, t  \,\to + \infty.
$$ 
Taking $\gamma = \sqrt{\mu}$ gives the best convergence rate, and the result of Theorem~\ref{strong-conv-thm}.} 
\end{remark}

\smallskip

\subsection{Non-smooth case}
Following \cite{AABR}, $\DIN{\gamma, \beta}$ is equivalent to the first-order system 
\[
\begin{cases}
\dot x(t) +  \beta \nabla f (x(t)) + \left( \gamma -\frac{1}{\beta}   \right)  x(t) +  \frac{1}{\beta} y(t) = 0;\\
 \dot{y}(t) + \left( \gamma -\frac{1}{\beta}   \right)  x(t) +  \frac{1}{\beta} y(t)  =0.
\end{cases} .
\]
This permits to extend $\DIN{\gamma, \beta}$ to the case of a proper lower semicontinuous convex function $f: \cH \to \Rb$. Replacing the gradient of $f$ by its subdifferential, we obtain its Non-Smooth version  :
\[
\DINNS{\gamma,\beta}  
\begin{cases}
\dot x(t) +  \beta \partial f (x(t)) + \left( \gamma -\frac{1}{\beta}   \right)  x(t) +  \frac{1}{\beta} y(t) \ni 0;\\
\dot{y}(t)    + \left( \gamma -\frac{1}{\beta}   \right)  x(t) +  \frac{1}{\beta} y(t)  =0.
\end{cases}
\]
Most  properties of  ${\rm (DIN)}_{\gamma, \beta}$ are still valid for this generalized version. To illustrate it, let us consider the following extension of Theorem~\ref{strong-conv-thm}.

\begin{theorem}\label{strong-conv-thm-gen}
Suppose that $f: \cH \to \Rb$ is lower semicontinuous and $\mu$-strongly convex for some $\mu >0$.
Let  $x(\cdot)$ be a  trajectory of $\DINNS{2\sqrt{\mu} ,\beta}$. Suppose that $ 0 \leq \beta \leq \frac{1}{2\sqrt{\mu}}$. Then
\begin{align*}
\frac{\mu}{2}\norm{x(t)-x^\star}^2 &\leq f(x(t))-  \min_{\cH}f =\cO \left( e^{-\frac{\sqrt{\mu}}{2}t} \right) \, \mbox{  as t } \,\to + \infty,  \\ 
&\mbox{ and } \, \int_{t_0}^{\infty} e^{\frac{\sqrt{\mu}}{2}t}\| \dot{x}(t)\|^2 dt < + \infty.
\end{align*}
\end{theorem}
\begin{proof}
Let us introduce  $\mathcal E : [t_0, +\infty[ \to \R^+   $ defined by
$$
\mathcal E (t)\eqdef f(x(t))-  \min_{\cH}f  + \frac{1}{2} \| \sqrt{\mu} (x(t) -x^\star) -  \left( 2 \sqrt{\mu} -\frac{1}{\beta}   \right)  x(t) -  \frac{1}{\beta} y(t)  \|^2,
$$
that will serve as a Lyapunov function. Then, the  proof follows the same lines as that of Theorem~\ref{strong-conv-thm}, with the use of the derivation rule of Brezis \cite[Lemme~3.3, p. 73]{Bre1}.
\end{proof}

\section{Inertial algorithms for strongly convex functions}\label{s_convex-algo}

\tcb{We will show in this section that the exponential convergence of Theorem~\ref{strong-conv-thm} for the inertial system \eqref{dyn-sc} translates into linear convergence in the algorithmic case under proper discretization.}

\subsection{Proximal algorithms}
\subsubsection{Smooth case}
Consider the inertial dynamic \eqref{dyn-sc}.
Its implicit discretization similar to that performed before gives
\begin{equation*}
\frac{1 }{h^2}(x_{k+1} - 2x_{k}+ x_{k-1} ) +   \frac{2\sqrt{\mu}}{h} (x_{k+1} - x_{k}) + \frac{\beta}{h} (\nabla f( x_{k+1}) - \nabla f( x_{k})   ) + \nabla f( x_{k+1})  =0,
\end{equation*} 
where $h$ is the positive step size. Set $s = h^2$. We obtain the following inertial proximal algorithm with hessian damping (SC refers to Strongly Convex):
\begin{eqnarray*}
\boxed{
\begin{array}{rcl}
&& {\rm \mbox{$\IPAHDSC$}}\\
\hline\\
&&  
\begin{cases}
y_k = x_{k} + \left( 1-{\frac{2\sqrt{\mu s}}{1 + 2 \sqrt{\mu s}}}\right) ( x_{k}  - x_{k-1})+ \beta \sqrt{s}   \left( 1-{\frac{2\sqrt{\mu s}}{1 + 2 \sqrt{\mu s}}}\right) \nabla f (x_k)   \\
x_{k+1} = \prox_{{\frac{\beta \sqrt{s} + s}{1 + 2 \sqrt{\mu s}}  f}}(y_k).
\end{cases}
\end{array}
}
\end{eqnarray*}

\begin{theorem} \label{prox-inertiel-strongly convex}
\tcb{Assume that $f: \cH \rightarrow \R$ is a convex $\mathcal C^1$ function and $\mu$-strongly convex, $\mu > 0$,} and suppose that 
\[
0 \leq \beta \leq \frac{1}{2\sqrt{\mu}}  \, \,  \mbox{ and } \, \, \sqrt{s}   \leq \beta.
\]
Set $q={\frac{1}{1+ \demi \sqrt{\mu s}}}$, which satisfies $0 <q <1$. Then, the sequence $\seq{x_k}$ generated by the algorithm {\IPAHDSC} obeys, for any $k \geq 1$
\[
\frac{\mu}{2}\norm{x_k-x^\star}^2 \leq f(x_k)-  \min_{\mathcal H}f \leq E_1 q^{k-1} ,
\]
where $E_1 = f(x_1) -f(x^\star) + \demi \|\sqrt{\mu} (x_1 - x^\star ) + \frac{1}{\sqrt{s}} (x_{1} - x_{0})
+ \beta \nabla f (x_1)\|^2$. Moreover, the gradients converge exponentially fast to zero: setting $\theta = \frac{1}{1+ \sqrt{\mu s}}$ which belongs to $]0,1[$, we have
$$
{\theta}^k \sum_{j=0}^{k-2} {\theta}^{-j} \|  \nabla f (x_{j})\|^2  = \cO \left(  q^k  \right) \, \mbox{  as k } \,\to + \infty.
$$ 
\end{theorem}

\tcb{
\begin{remark}
We are not aware of any result of this kind for such a proximal algorithm. 
\end{remark}
}

\begin{proof}
\tcb{Let $x^\star$ be the unique minimizer of $f$}, and consider the sequence   $\seq{E_k}$ 
$$
E_k\eqdef f(x_k) -f(x^\star) + \demi \|v_k\|^2 ,
$$
where \, $
v_k = \sqrt{\mu} (x_k - x^\star ) + \frac{1}{\sqrt{s}} (x_{k} - x_{k-1})
+ \beta \nabla f (x_k).$

\smallskip

\noindent We will use the following equivalent formulation of the algorithm \IPAHDSC
\begin{equation}\label{basic-prox-bbbb}
\frac{1 }{\sqrt{s}}(x_{k+1} - 2x_{k}+ x_{k-1} ) +   2\sqrt{\mu} (x_{k+1} - x_{k}) + \beta (\nabla f( x_{k+1}) - \nabla f( x_{k})   ) + \sqrt{s}\nabla f( x_{k+1})  =0.
\end{equation}
We have
\begin{eqnarray*}
E_{k+1}  - E_k &=&  f (x_{k+1}) -f (x_k) + \frac{1}{2}\|v_{k+1}\|^2 -\frac{1}{2}\|v_k\|^2 .
 \end{eqnarray*}
 Using successively the definition of $v_k$ and \eqref{basic-prox-bbbb}, we get
\begin{eqnarray*}
v_{k+1} - v_{k}&=& \sqrt{\mu}  (x_{k+1} - x_{k}) +\frac{1}{\sqrt{s}} (x_{k+1} - 2x_{k} + x_{k-1})+ \beta (\nabla f (x_{k+1}) - \nabla f (x_k))\\
&=& \sqrt{\mu}  (x_{k+1} - x_{k}) -2\sqrt{\mu} (x_{k+1} - x_{k}) - \sqrt{s}\nabla f( x_{k+1}) \\
&=&  =  -\sqrt{\mu} (x_{k+1} - x_{k}) - \sqrt{s}\nabla f( x_{k+1}).
\end{eqnarray*}
Write shortly $B_k = \sqrt{\mu} (x_{k+1} - x_{k}) + \sqrt{s}\nabla f( x_{k+1})$. We have
\begin{eqnarray*}
&&\frac{1}{2}\|v_{k+1}\|^2 -\frac{1}{2}\|v_k\|^2  = \left\langle  v_{k+1} -v_{k} , v_{k+1} \right\rangle - \frac{1}{2}\|v_{k+1} - v_{k}\|^2 \\
&& =- \left\langle  B_k , \sqrt{\mu} (x_{k+1} - x^\star ) + \frac{1}{\sqrt{s}} (x_{k+1} - x_{k})
+ \beta \nabla f (x_{k+1}) \right\rangle 
- \frac{1}{2}\|B_k\|^2  
\\
&& = -\mu\left\langle x_{k+1} - x_{k}, x_{k+1} -x^\star \right\rangle 
-\sqrt{\frac{\mu}{s}}\|  x_{k+1} - x_{k} \|^2 - \beta \sqrt{\mu}
\left\langle  \nabla f( x_{k+1}), x_{k+1} -x_k \right\rangle   \\
&& - \sqrt{\mu s} \left\langle  \nabla f( x_{k+1}),x_{k+1} -x^\star \right\rangle  - \left\langle  \nabla f( x_{k+1}), x_{k+1} -x_k \right\rangle
- \beta \sqrt{s} \| \nabla f( x_{k+1}) \|^2\\
&&- \frac{1}{2}\mu \|x_{k+1} - x_{k}\|^2 - \frac{1}{2}s \|\nabla f( x_{k+1}\|^2  - \sqrt{\mu s} \left\langle  \nabla f( x_{k+1}), x_{k+1} -x_k \right\rangle
\end{eqnarray*}
By virtue of strong convexity of $f$
\begin{eqnarray*}
f(x_{k}) &\geq & f( x_{k+1}) + \left\langle  \nabla f( x_{k+1}), x_{k}- x_{k+1}  \right\rangle + \frac{\mu}{2}\|x_{k+1} - x_{k}\|^2 ;\\
f(x^\star) &\geq & f( x_{k+1}) + \left\langle  \nabla f( x_{k+1}), x^\star- x_{k+1}  \right\rangle + \frac{\mu}{2}\|x_{k+1} - x^\star\|^2 .
\end{eqnarray*}
Combining the above results,  and after dividing by $\sqrt{s}$, we get 
\begin{eqnarray*}
&&\frac{1}{\sqrt{s}} (E_{k+1}  - E_k) + \sqrt{\mu}[ f( x_{k+1}) - f(x^\star)+ \frac{\mu}{2}\|x_{k+1} - x^\star\|^2     ]\\
&&\leq -\frac{\mu}{\sqrt{s}}\left\langle x_{k+1} - x_{k}, x_{k+1} -x^\star \right\rangle 
-\frac{\sqrt{\mu}}{s}\|  x_{k+1} - x_{k} \|^2 \\
&&- \beta \sqrt{\frac{\mu}{s}}
\left\langle  \nabla f( x_{k+1}), x_{k+1} -x_k \right\rangle   
 -\frac{\mu}{2\sqrt{s}}\|x_{k+1} - x_{k}\|^2 - \beta \| \nabla f( x_{k+1}) \|^2\\
&&
- \frac{\mu}{2\sqrt{s}}\|x_{k+1} - x_{k}\|^2 - \frac{1}{2}\sqrt{s} \|\nabla f( x_{k+1}\|^2  - \sqrt{\mu } \left\langle  \nabla f( x_{k+1}), x_{k+1} -x_k \right\rangle ,
\end{eqnarray*}
which gives, after developing and simplification
\begin{eqnarray*}
&&\frac{1}{\sqrt{s}} (E_{k+1}  - E_k) + \sqrt{\mu}E_{k+1} - \beta \mu 
\left\langle  \nabla f( x_{k+1}), x_{k+1} -x^\star \right\rangle \\
&& \leq -\left(\frac{\sqrt{\mu}}{2s} + \frac{\mu }{\sqrt{s}} \right) \|x_{k+1} - x_{k}\|^2  - \left( \beta  -\frac{\beta^2\sqrt{\mu}}{2} +\frac{\sqrt{s}}{2} \right) \|\nabla f( x_{k+1})\|^2\\
&& - \sqrt{\mu } \left\langle  \nabla f( x_{k+1}), x_{k+1} -x_k \right\rangle .
\end{eqnarray*}
According to $ 0 \leq \beta \leq \frac{1}{2\sqrt{\mu}}$, we have 
$\beta  -\frac{\beta^2\sqrt{\mu}}{2} \geq \frac{3\beta}{4}$, which, with Cauchy-Schwarz inequality, gives
\begin{eqnarray*}
&&\frac{1}{\sqrt{s}} (E_{k+1}  - E_k) 
+ \sqrt{\mu} E_{k+1} +\left(\frac{\sqrt{\mu}}{2s} + \frac{\mu }{\sqrt{s}} \right) \|x_{k+1} - x_{k}\|^2  +  \frac{3\beta}{4} \|\nabla f( x_{k+1})\|^2 \\
&& - \beta \mu 
\| \nabla f( x_{k+1})\| \| x_{k+1} -x^\star \|  - \sqrt{\mu } \| \nabla f( x_{k+1})\| \| x_{k+1} -x_k \| \leq 0.
\end{eqnarray*}
Let us  use again the strong convexity of $f$ to write
$$
E_{k+1}   \geq \frac{1}{2}E_{k+1} + \frac{1}{2} \left(f(x_{k+1})- f(x^\star)\right) \geq 
\frac{1}{2}E_{k+1}  +  \frac{\mu}{4} \| x_{k+1} -x^\star \|^2 .
$$
Combining the two inequalities above, we get
\begin{eqnarray*}
&&\frac{1}{\sqrt{s}} (E_{k+1}  - E_k) 
+ \demi\sqrt{\mu} E_{k+1} + \sqrt{\mu} \frac{\mu}{4} \| x_{k+1} -x^\star \|^2  + \left(\frac{\sqrt{\mu}}{2s} + \frac{\mu }{\sqrt{s}} \right) \|x_{k+1} - x_{k}\|^2 \\
&&  +  \frac{3\beta}{4} \|\nabla f( x_{k+1})\|^2  - \beta \mu 
\| \nabla f( x_{k+1})\| \| x_{k+1} -x^\star \|  - \sqrt{\mu } \| \nabla f( x_{k+1})\| \| x_{k+1} -x_k \| \leq 0 .
\end{eqnarray*}
Let us rearrange  the terms as follows
\begin{eqnarray*}
&&\frac{1}{\sqrt{s}} (E_{k+1}  - E_k) 
+ \demi\sqrt{\mu} E_{k+1} \\
&&+ 
\underset{\text{Term~1}}{\underbrace{\left( \sqrt{\mu} \frac{\mu}{4} \| x_{k+1} -x^\star \|^2   + \frac{\beta}{2} \|\nabla f( x_{k+1})\|^2  - \beta \mu \| \nabla f( x_{k+1})\| \| x_{k+1} -x^\star \|\right)}} +\\
&&  
\underset{\text{Term~2}}{\underbrace{\left( \left(\frac{\sqrt{\mu}}{2s} + \frac{\mu }{\sqrt{s}} \right)  \|x_{k+1} - x_{k}\|^2 + \frac{\beta}{4} \|\nabla f( x_{k+1})\|^2  - \sqrt{\mu } \| \nabla f( x_{k+1})\| \| x_{k+1} -x_k \|   \right)}} \leq 0
\end{eqnarray*}

Let us examine the sign of the last two terms in the rhs of inequality above.
\begin{enumerate}[label={Term~}\arabic*]
\item Set $X=\| x_{k+1} -x^\star\|$, $Y=  \| \nabla f( x_{k+1}) \|$. Elementary algebra gives that
$$
\sqrt{\mu}\frac{\mu}{4} X^2 + \frac{\beta}{2} Y^2 -  \beta\mu XY \geq 0,
$$
holds true  under the condition $ 0 \leq \beta \leq \frac{1}{2\sqrt{\mu}}$. Hence, under this condition 
$$
\sqrt{\mu} \frac{\mu}{4} \| x_{k+1} -x^\star \|^2   + \frac{\beta}{2} \|\nabla f( x_{k+1})\|^2  - \beta \mu 
\| \nabla f( x_{k+1})\| \| x_{k+1} -x^\star \| \geq 0.
$$
\item Set $X=\norm{x_{k+1} - x_{k}} $, $Y=\norm{\nabla f( x_{k+1})}$. Elementary algebra gives
$$
\left(\frac{\sqrt{\mu}}{2s} + \frac{\mu }{\sqrt{s}} \right)  X^2 + 
\frac{\beta}{4} Y^2  - \sqrt{\mu } XY \geq 0 
$$
holds true  under the condition $ \frac{\sqrt{\mu}}{2s} + \frac{\mu }{\sqrt{s}} \geq \frac{\mu}{\beta}$. Hence, under this condition 
$$
\left(\frac{\sqrt{\mu}}{2s} + \frac{\mu }{\sqrt{s}} \right)  \|x_{k+1} - x_{k}\|^2 + 
\frac{\beta}{4} \|\nabla f( x_{k+1})\|^2  - \sqrt{\mu } \| \nabla f( x_{k+1})\| \| x_{k+1} -x_k \|  \geq 0.
$$
In turn, the condition $ \frac{\sqrt{\mu}}{2s} + \frac{\mu }{\sqrt{s}} \geq \frac{\mu}{\beta}$ is equivalent to
$
\sqrt{s}   \leq \frac{\beta}{2}\left( 1 +  \sqrt{1+ \frac{2}{\beta \sqrt{\mu}}}  \right).$\\
Clearly, this condition is satisfied if $\sqrt{s} \leq \beta $.
\end{enumerate}

Let us put the above results together. We have obtained that, under the conditions $ 0 \leq \beta \leq \frac{1}{2\sqrt{\mu}}$ and $\sqrt{s}   \leq \beta$,
$$
\frac{1}{\sqrt{s}} (E_{k+1}  - E_k) 
+ \demi\sqrt{\mu} E_{k+1}  \leq 0.
$$
Set $q= \frac{1}{1+ \demi \sqrt{\mu s}}$, which satisfies $0 <q <1$. From this, we infer
$
E_k \leq q E_{k-1}
$
which gives 
\begin{equation}\label{exp-decrease-1}
E_k \leq E_1 q^{k-1}.
\end{equation}
Since $E_k \geq  f(x_k) - f(x^\star)$, we finally obtain
$$
f(x_k) -f(x^\star) \leq E_1 q^{k-1} = \cO \pa{q^{k}} .
$$

Let us now estimate the convergence rate of the gradients to zero.
 According to the exponential decay of $\seq{E_k}$, as given in \eqref{exp-decrease-1}, and by definition of $E_k$, we have, for  all $k\geq 1$
$$
\| \sqrt{\mu} (x_k - x^\star ) + \frac{1}{\sqrt{s}} (x_{k} - x_{k-1})
+ \beta \nabla f (x_k)\|^2  \leq 2 E_{k} \leq 2 E_1 q^{k-1}.
$$
After developing, we get
\begin{eqnarray*}
&&\mu \| x_k -x^\star\|^2 + \frac{1}{s}\|  x_k - x_{k-1}\|^2 + \beta^2  \|  \nabla f (x_k)\|^2    + 2\beta \sqrt{\mu}\left\langle x_k -x^\star, \nabla f (x_k)  \right\rangle \\
&& +  \frac{1}{\sqrt{s}} \left\langle  x_k - x_{k-1},  2 \beta\nabla f ( x_k ) +  2\sqrt{\mu}  (  x_k -x^\star)\right\rangle  \leq 2 E_1 q^{k-1} .
\end{eqnarray*}
By  convexity of $f$, we have 
$$\left\langle x_k -x^\star, \nabla f (x_k)  \right\rangle \geq f(x_k) - f(x^\star) \,  \mbox{and} \,
\left\langle x_k - x_{k-1},  \nabla f (x_k) \right\rangle \geq f(x_k) - f(x_{k-1})$$
 Moreover, 
$ \left\langle  x_k - x_{k-1},   x_k -x^\star\right\rangle \geq \demi \| x_k -x^\star\|^2 - \demi \| x_{k-1} -x^\star\|^2$.\\
Combining the above results, we obtain
\begin{eqnarray*}
&&\sqrt{\mu} \pa{2 \beta (f(x_k) - f(x^\star)) +\sqrt{\mu}\norm{ x_k -x^\star}^2} + \beta^2 \norm{\nabla f (x_k)}^2     \\
&& + \frac{1}{\sqrt{s}} \pa{2 \beta (f(x_k) - f(x^\star)) + \sqrt{\mu} \norm{x_k -x^\star}^2}\\
&& - \frac{1}{\sqrt{s}} \pa{2 \beta (f(x_{k-1}) - f(x^\star)) + \sqrt{\mu} \norm{x_{k-1} -x^\star}^2} \leq 2 E_1 q^{k-1}.
\end{eqnarray*}
Set $Z_k\eqdef 2 \beta (f(x_k) - f(x^\star)    )+ \sqrt{\mu} \| x_k -x^\star\|^2$. We have, for all $k\geq 1$
\begin{equation}\label{exp-decrease-2}
\frac{1}{\sqrt{s}} \left( Z_k - Z_{k-1}    \right) + \sqrt{\mu} Z_k +   \beta^2  \|  \nabla f (x_k)\|^2   \leq  2E_1 q^{k-1}.
\end{equation}
Set $\theta = \frac{1}{1+ \sqrt{\mu s}}$ which belongs to $]0,1[$. Equivalently
$$
Z_k + \theta \beta^2  \sqrt{s} \|  \nabla f (x_k)\|^2  \leq\theta  Z_{k-1} +2E_1 \theta \sqrt{s}  q^{k-1}.
$$
Iterating this linear recursive inequality gives 
\begin{equation}\label{exp-decrease-3}
Z_k + \theta \beta^2  \sqrt{s} \sum_{p=0}^{k-2} {\theta}^p \|  \nabla f (x_{k-p})\|^2  \leq  \theta^{k-1} Z_1  +2E_1 \theta \sqrt{s} \sum_{p=0}^{k-2} {\theta}^p q^{k-p-1}.
\end{equation}
Then notice that $\frac{\theta}{q}= \frac{1+ \demi \sqrt{\mu s}}{1+ \sqrt{\mu s}}  <1  $, which gives
 $$\sum_{p=0}^{k-2} {\theta}^p q^{k-p-1}= {q}^{k-1} \sum_{p=0}^{k-2}\left(\frac{\theta}{q} \right)^p \leq 2 \left(1+ \frac{1}{\sqrt{\mu s}}    \right) {q}^{k-1}.  $$
Collecting the above results, we obtain
\begin{equation}\label{exp-decrease-4}
\theta \beta^2  \sqrt{s} \sum_{p=0}^{k-2} {\theta}^p \|  \nabla f (x_{k-p})\|^2  \leq 
\theta^{k-1} Z_1 + \frac{4E_1}{\sqrt{\mu }}  {q}^{k-1} .
\end{equation}
Using again the inequality $\theta <q$, and after reindexing, we finally obtain
$$
{\theta}^k \sum_{p=0}^{k-2} {\theta}^{-j} \|  \nabla f (x_{j})\|^2  = \cO \pa{q^k}.
$$
\qed
\end{proof}
\subsubsection{Non-smooth case}
Let $f: \cH \to \R \cup \left\lbrace +\infty \right\rbrace$ be a proper, lower semicontinuous and convex function. We argue as in Section~\ref{non-smooth-prox} by replacing $f$ with its Moreau envelope $f_{\lambda} $. \tcb{The key observation is that the Moreau-Yosida regularization also preserves strong convexity, though with a different modulus as shown by the following result.}
\begin{proposition} Suppose that $f: \cH \to \Rb $ is a proper, lower semicontinuous convex function.
Then, for any $\lambda >0$ and $\mu >0$
\[
f \, \mbox{is} \, \,  \text{$\mu$-strongly convex}\, \Longrightarrow \,  f_{\lambda} \mbox{is strongly convex with modulus } \, \frac{\mu}{1 + \lambda \mu}.
\]
\end{proposition}
\begin{proof}
 If $f$ is  strongly convex with constant $\mu >0$, we have $f= g + \frac{\mu}{2} \|\cdot\|^2$  for some convex function~$g$. Elementary calculus (see e.g., \cite[Exercise 12.6]{BC}) gives, with $\theta= \frac{\lambda }{1+ \lambda \mu}  $,
$$
f_{\lambda}(x) = g_{\theta}\left( \frac{1 }{1 +\lambda \mu}\,x\right) + \frac{\mu}{2(1 + \lambda \mu)}\| x\|^2.  
$$
Since $x \mapsto g_{\theta}\left( \frac{1 }{1 + \lambda\mu }\,x\right) $ is  convex, the above formula shows that $f_{\lambda}$ is strongly convex with  constant $\frac{\mu}{1 + \lambda\mu}$. \qed
\end{proof}

\tcb{According to the expressions \eqref{eq:envgrad} and \eqref{eq:envprox}}, \IPAHDSC becomes with $\theta={\frac{\beta \sqrt{s} + s}{1 + 2 \sqrt{\frac{\mu}{1 + \lambda\mu} s}}}$ and $ 
a = {\frac{2\sqrt{\frac{\mu}{1 + \lambda\mu} s}}{1 + 2 \sqrt{\frac{\mu}{1 + \lambda\mu} s}}}$:
\begin{eqnarray*} 
\boxed{
\begin{array}{rcl}
&& {\IPAHDNSSC} \\
\hline\\ 
&&
\begin{cases}
y_k=   x_{k} + (1-a) ( x_{k}  - x_{k-1}) + \frac{\beta \sqrt{s}}{\lambda} (1-a)  \left( x_k-\prox_{ \lambda f} (x_k)\right)   
\\
x_{k+1} =  \frac{\lambda}{\lambda +\theta}y_k + \frac{\theta}{\lambda +\theta}\prox_{ (\lambda + \theta)f}(y_k) 
\end{cases}
\vspace{2mm}
\end{array}
}
\end{eqnarray*}

It is a relaxed inertial proximal algorithm whose coefficients are constant. As a result, its computational burden is equivalent to (actually twice) that of the classical proximal algorithm. A direct application of  the conclusions of Theorem~\ref{prox-inertiel-strongly convex} to $f_{\lambda}$ gives the following statement.

\begin{theorem} 
Suppose that $f: \cH \to \Rb $ is a proper, lower semicontinuous and convex function which is $\mu$-strongly convex for some $\mu >0$. Take $\lambda >0$.
Suppose that 
$$ 
0 \leq \beta \leq \frac{1}{2} \sqrt{ \lambda + \frac{1}{\mu}}  \, \,  \mbox{ and } \, \, \sqrt{s}   \leq \beta.$$
Set  $q=\displaystyle{\frac{1}{1+ \demi \sqrt{\frac{\mu}{1 + \lambda \mu} s}}}$, which satisfies $0 <q <1$.
Then, for any sequence $\seq{x_k}$ generated by algorithm \IPAHDNSSC 
$$
\norm{x_k-x^\star} = \cO\pa{q^{k/2}} \quad \text{and} \quad f(\prox_{ \lambda f}(x_k))-  \min_{\cH}f = \cO\pa{q^{k}} \, \mbox{  as k } \,\to + \infty,$$
and 
$$
  \| x_{k} - \prox_{ \lambda f}(x_{k})\|^2  = \cO \left( q^{k} \right) \, \mbox{  as k } \,\to + \infty.
$$
\end{theorem}

\subsection{Inertial gradient algorithms}
Let us embark from the continuous dynamic \eqref{dyn-sc} whose linear convergence rate was established in Theorem \ref{strong-conv-thm}. Its explicit time discretization with centered finite differences for speed and acceleration gives
\begin{equation*}
\frac{1}{s}(x_{k+1} - 2x_{k}+ x_{k-1} ) +  \frac{\sqrt{\mu}}{\sqrt{s}}(x_{k+1} - x_{k-1}) + \beta \frac{1}{\sqrt{s}}(\nabla f( x_{k}) - \nabla f( x_{k-1})   ) + \nabla f( x_{k})  = 0 .
\end{equation*}
Equivalently,
 \begin{equation}\label{basic-grad-b}
(x_{k+1} - 2x_{k}+ x_{k-1} ) +   \sqrt{\mu s} (x_{k+1} - x_{k-1}) + \beta \sqrt{s}(\nabla f( x_{k}) - \nabla f( x_{k-1})   ) + s\nabla f( x_{k})  =0,
\end{equation}
which gives the inertial gradient algorithm with Hessian damping (SC stands for Strongly Convex):
\begin{eqnarray*}
\boxed{
\begin{array}{rcl}
&& {\IGAHDSC} \\
\hline
\smallskip\\ 
&& x_{k+1}=    x_{k} + \frac{1 -\sqrt{\mu s}  }{1 +\sqrt{\mu s}}( x_{k}  - x_{k-1})-
\frac{\beta \sqrt{s} }{1 +\sqrt{\mu s}}  \left(\nabla f (x_k)  - \nabla f (x_{k-1})  \right) \hspace{1cm}\vspace{2mm} \\
&& \qquad \quad -\frac{s}{1 +\sqrt{\mu s}}\nabla f (x_k). 
\end{array}
}
\end{eqnarray*}
Let us analyze the linear convergence rate of \IGAHDSC.

\begin{theorem} \label{grad-inertiel-strongly convex}
\tcb{Let $f: \cH \to \R$ be a $\mathcal{C}^1$ and $\mu$-strongly convex function for some $\mu >0$, and whose gradient $\nabla f$  is $L$-Lipschitz continuous.} Suppose that
\begin{align}\label{eq:IGAHDSCcondsbeta}
\beta \leq \displaystyle{ \frac{1}{\sqrt{\mu}}} \text{ and } L \leq \min \left\lbrace \displaystyle{\frac{\sqrt{\mu}}{8 \beta},  \frac{\frac{\sqrt{\mu}}{2s} + \frac{\mu }{\sqrt{s}}}
{2\beta \mu + \frac{1 }{\sqrt{s}} + \frac{\sqrt{\mu}}{2}} } \right\rbrace .
\end{align}
Set  $q= \displaystyle{\frac{1}{1+ \demi \sqrt{\mu s}}}$, which satisfies $0 <q <1$.
Then, for any sequence $\seq{x_k}$ generated by algorithm {\IGAHDSC}, we have
\[
\norm{x_k-x^\star} = \cO\pa{q^{k/2}} \quad \text{and} \quad f(x_k)-  \min_{\cH}f =\cO \left( q^{k} \right) \, \mbox{  as k } \,\to + \infty.
\]
Moreover, the gradients converge exponentially fast to zero: setting $\theta = \frac{1}{1+ \sqrt{\mu s}}$ which belongs to $]0,1[$, we have
$$
{\theta}^k \sum_{p=0}^{k-2} {\theta}^{-j} \|  \nabla f (x_{j})\|^2  = \cO \left(  q^k  \right) \, \mbox{  as k } \,\to + \infty.
$$ 
\end{theorem}

\tcb{
\begin{remark}
{~}
\begin{enumerate}[label=\arabic*.]
\item {\IGAHDSC} can be seen as an extension of the Nesterov accelerated method for strongly convex functions that corresponds to the particular case $\beta=0$. Actually, in this very specific case, {\IGAHDSC} is nothing but the (HBF) method with stepsize parameter $a = \frac{s}{1 +\sqrt{\mu s}}$ and momentum parameter $b = \frac{1 - \sqrt{\mu s}}{1 +\sqrt{\mu s}}$; see \cite[(2) in Section~3.2]{Polyak2}. Thus, if $f$ is also of class $\mathcal{C}^2$ at $x^\star$, one can obtain linear convergence of the iterates $\seq{x_k}$ (but not the objective values) from \cite[Theorem~1]{Polyak2} under the assumption that $s < 4/L$ (which can be shown to be weaker than \eqref{eq:IGAHDSCcondsbeta} since the latter is equivalent for $\beta=0$ to $sL \leq (\sqrt{1-c+c^2}-(1-c))^2/c \leq 1$, where $c=\mu/L$).
\item In fact, even for $\beta > 0$, by lifting the problem to the vector $z_k = \begin{pmatrix} x_{k} - x^\star \\ x_{k-1} - x^\star\end{pmatrix}$ as is standard in the (HBF) method, one can write {\IGAHDSC} as
\[
z_{k+1} = 
\begin{pmatrix} 
(1+b)\Id  - (a+d)\nabla f^2(x^\star) & -b \Id + d\nabla f^2(x^\star) \\ 
\Id & 0
\end{pmatrix} 
z_{k} + o(z_k) ,
\]
where $d=\frac{\beta \sqrt{s} }{1 +\sqrt{\mu s}}$. Linear convergence of the iterates $\seq{x_k}$ can then be obtained by studying the spectral properties of the above matrix.

\item For $\beta=0$, Theorem~\ref{grad-inertiel-strongly convex} recovers \cite[Theorem~3.2]{Siegel}, though the author uses a slightly different discretization, requires only $s \leq 1/L$ and his convergence rate is $\displaystyle{\pa{1+\sqrt{\mu s}}}^{-1}$, which is slightly better than ours for this special case. In the case $\beta > 0$,  a result on a scheme related but different from {\IGAHDSC} can be found in \cite[Theorem~3]{WRJ} (their rate is also slightly worse than ours). Our estimate are also new in the literature.
\end{enumerate}
\end{remark}
}

\begin{proof}
The proof is based on Lyapunov analysis, and the decrease property at linear rate of the sequence $\seq{E_k}$ defined by
$$
E_k\eqdef f( x_{k}) -f(x^\star) + \demi \|v_k\|^2 ,
$$
where  $x^\star$ is the unique minimizer of $f$, and 
$$
v_k = \sqrt{\mu} (x_{k-1} - x^\star ) + \frac{1}{\sqrt{s}} (x_{k} - x_{k-1})
+ \beta \nabla f (x_{k-1}).
$$
We have
$E_{k+1}  - E_k =  f (x_{k+1}) -f (x_{k}) + \frac{1}{2}\|v_{k+1}\|^2 -\frac{1}{2}\|v_k\|^2 .$ 
Using successively the definition of $v_k$ and \eqref{basic-grad-b}, we obtain
\begin{eqnarray*}
&&v_{k+1} - v_{k}= \sqrt{\mu}  (x_{k} - x_{k-1}) +\frac{1}{\sqrt{s}} (x_{k+1} - 2x_{k} + x_{k-1})+ \beta (\nabla f (x_{k}) - \nabla f (x_{k-1}))\\
&&= \frac{1}{\sqrt{s}} \Big((x_{k+1} - 2x_{k} + x_{k-1})+ \sqrt{\mu s} (x_{k} - x_{k-1}) +   \beta \sqrt{s}(\nabla f( x_{k}) - \nabla f( x_{k-1})   ) \Big) \\
&&=   \frac{1}{\sqrt{s}} \Big(-s\nabla f( x_{k})- \sqrt{\mu s} (x_{k+1} - x_{k-1}) +   \sqrt{\mu s} (x_{k} - x_{k-1})    ) \Big)\\
&&=    -\sqrt{\mu} (x_{k+1} - x_{k}) - \sqrt{s}\nabla f( x_{k}).
\end{eqnarray*}
Since $\frac{1}{2}\|v_{k+1}\|^2 -\frac{1}{2}\|v_k\|^2  = \left\langle  v_{k+1} -v_{k} , v_{k+1} \right\rangle - \frac{1}{2}\|v_{k+1} - v_{k}\|^2 $, we have
\begin{eqnarray*}
&&\frac{1}{2}\|v_{k+1}\|^2 -\frac{1}{2}\|v_k\|^2  = - \frac{1}{2}\|\sqrt{\mu} (x_{k+1} - x_{k}) + \sqrt{s}\nabla f( x_{k})\|^2 \\
&&- \left\langle\sqrt{\mu} (x_{k+1} - x_{k}) + \sqrt{s}\nabla f( x_{k}) , \sqrt{\mu} (x_{k} - x^* ) + \frac{1}{\sqrt{s}} (x_{k+1} - x_{k})
+ \beta \nabla f (x_{k}) \right\rangle 
\\
&& = -\mu\left\langle x_{k+1} - x_{k}, x_{k} -x^* \right\rangle 
-\sqrt{\frac{\mu}{s}}\| x_{k+1} - x_{k} \|^2  - \beta \sqrt{\mu}
\left\langle  \nabla f( x_{k}), x_{k+1} -x_{k}\right\rangle   \\
&& - \sqrt{\mu s} \left\langle  \nabla f( x_{k}),x_{k} -x^* \right\rangle  - \left\langle  \nabla f( x_{k}), x_{k+1} -x_k \right\rangle
- \beta \sqrt{s} \| \nabla f( x_{k})\|^2 \\
&&- \frac{1}{2}\mu \|x_{k+1} - x_{k}\|^2 - \frac{1}{2}s \|\nabla f( x_{k}\|^2  - \sqrt{\mu s} \left\langle  \nabla f( x_{k}), x_{k+1} -x_{k} \right\rangle .
\end{eqnarray*}
By strong convexity of $f$ and $L$-Lipschitz continuity of $\nabla f $ we have  
\begin{eqnarray*}
f(x^\star) &\geq&  f( x_{k}) + \left\langle  \nabla f( x_{k}), x^\star- x_{k}  \right\rangle + \frac{\mu}{2}\| x_{k}  - x^\star\|^2\\
f(x_{k}) &\geq & f( x_{k+1}) + \left\langle  \nabla f( x_{k+1}), x_{k}- x_{k+1}  \right\rangle + \frac{\mu}{2}\|x_{k+1} - x_{k}\|^2 \\
&\geq & f( x_{k+1}) + \left\langle  \nabla f( x_{k}), x_k- x_{k+1}  \right\rangle + (\frac{\mu}{2} -L)\|x_{k+1} - x_k\|^2 .
\end{eqnarray*} 
Combining the results above, and after dividing by $\sqrt{s}$, we get
\begin{eqnarray*}
&&\frac{1}{\sqrt{s}} (E_{k+1}  - E_k) + \sqrt{\mu}[ f( x_{k+1}) - f(x^\star)+ \frac{\mu}{2}\|x_{k} - x^\star\|^2  ]  + \sqrt{\mu}(f( x_{k}) -f( x_{k+1}))\\
&&\leq -\frac{\mu}{\sqrt{s}}\left\langle x_{k+1} - x_{k}, x_{k} -x^\star \right\rangle 
-\frac{\sqrt{\mu}}{s}\|x_{k+1} - x_{k}\|^2 - \beta \sqrt{\frac{\mu}{s}}
\left\langle  \nabla f( x_{k}), x_{k+1} -x_{k} \right\rangle \\
&& + \frac{1}{\sqrt{s}}(L-\frac{\mu}{2})\|x_{k+1} - x_{k}\|^2
- \frac{\mu}{2\sqrt{s}}\|x_{k+1} - x_{k}\|^2 \\
&&-\left( \beta + \frac{1}{2}\sqrt{s} \right)\|\nabla f( x_{k}\|^2  - \sqrt{\mu } \left\langle  \nabla f( x_{k}), x_{k+1} -x_{k} \right\rangle .
\end{eqnarray*}
Let us make appear $E_k$
\begin{eqnarray*}
&&\frac{1}{\sqrt{s}} (E_{k+1}  - E_k) + \sqrt{\mu}E_{k+1}  \leq  \sqrt{\mu} \left\langle  \nabla f( x_{k}), x_{k+1} -x_{k} \right\rangle + \sqrt{\mu} \frac{L}{2}\|x_{k+1} - x_{k}\|^2 \\
&& + \frac{\sqrt{\mu}}{2}\|\frac{1}{\sqrt{s}} (x_{k+1} - x_{k})
+ \beta \nabla f (x_{k})   \|^2 +\mu\left\langle x_{k} -x^\star, \frac{1}{\sqrt{s}} (x_{k+1} - x_{k})
+ \beta \nabla f (x_{k})   \right\rangle \\
&& -\frac{\mu}{\sqrt{s}}\left\langle x_{k+1} - x_{k}, x_{k} -x^\star \right\rangle 
-\frac{\sqrt{\mu}}{s}\|x_{k+1} - x_{k}\|^2 - \beta \sqrt{\frac{\mu}{s}}
\left\langle  \nabla f( x_{k}), x_{k+1} -x_{k} \right\rangle \\
&& + \frac{1}{\sqrt{s}}(L-\frac{\mu}{2})\|x_{k+1} - x_{k}\|^2
- \frac{\mu}{2\sqrt{s}}\|x_{k+1} - x_{k}\|^2 \\
&&-\left( \beta + \frac{1}{2}\sqrt{s} \right)\|\nabla f( x_{k}\|^2  - \sqrt{\mu } \left\langle  \nabla f( x_{k}), x_{k+1} -x_{k} \right\rangle .
\end{eqnarray*}
After developing and simplification, we get
\begin{eqnarray*}
&&\frac{1}{\sqrt{s}} (E_{k+1}  - E_k) + \sqrt{\mu}E_{k+1} 
 \leq -\left(\frac{\sqrt{\mu}}{2s} + \frac{\mu }{\sqrt{s}} -L\left(             \frac{1 }{\sqrt{s}} + \frac{\sqrt{\mu}}{2}\right)\right) \|x_{k+1} - x_{k}\|^2 \\
 && - \left( \beta  -\frac{\beta^2\sqrt{\mu}}{2} +\frac{\sqrt{s}}{2} \right) \|\nabla f( x_{k+1})\|^2 + \beta \mu  
\left\langle  \nabla f( x_{k}), x_{k} -x^\star \right\rangle . 
\end{eqnarray*}
Let us majorize this last term by using the Lipschitz continuity of 
$\nabla f$
\begin{eqnarray*}
\left\langle  \nabla f( x_{k}), x_{k} -x^\star \right\rangle &=&   \left\langle  \nabla f( x_{k}) -  \nabla f( x^\star), x_{k} -x^\star \right\rangle
 \leq  L \|x_{k} -x^\star\| ^2 \\
&\leq & 2L \|x_{k+1} -x^\star\| ^2  + 2L \|x_{k+1} - x_{k}\|^2 .
\end{eqnarray*}
Therefore
\begin{eqnarray*}
&&\frac{1}{\sqrt{s}} (E_{k+1}  - E_k) + \sqrt{\mu}E_{k+1} 
+\left(\frac{\sqrt{\mu}}{2s} + \frac{\mu }{\sqrt{s}} -L\left(             2\beta \mu + \frac{1 }{\sqrt{s}} + \frac{\sqrt{\mu}}{2}\right)\right) \|x_{k+1} - x_{k}\|^2 \\
 && +\left( \beta  -\frac{\beta^2\sqrt{\mu}}{2} +\frac{\sqrt{s}}{2} \right) \|\nabla f( x_{k+1})\|^2 - 2\beta \mu  L 
\|x_{k+1} - x^\star\|^2 \leq 0.  
\end{eqnarray*}
According to $ 0 \leq \beta \leq \frac{1}{\sqrt{\mu}}$, we have 
$\beta  -\frac{\beta^2\sqrt{\mu}}{2} \geq \frac{\beta}{2}$, which gives
\begin{eqnarray*}
&&\frac{1}{\sqrt{s}} (E_{k+1}  - E_k) + \sqrt{\mu}E_{k+1} 
+\left(\frac{\sqrt{\mu}}{2s} + \frac{\mu }{\sqrt{s}} -L\left(             2\beta \mu + \frac{1 }{\sqrt{s}} + \frac{\sqrt{\mu}}{2}\right)\right) \|x_{k+1} - x_{k}\|^2 \\
 && +\frac{\beta}{2} \|\nabla f( x_{k+1})\|^2 - 2\beta \mu  L 
\|x_{k+1} - x^\star\|^2 \leq 0.  
\end{eqnarray*}
Let us  use again the strong convexity of $f$ to write
$$
E_{k+1}   \geq \frac{1}{2}E_{k+1} + \frac{1}{2} \left(f(x_{k+1})- f(x^\star)\right) \geq 
\frac{1}{2}E_{k+1}  +  \frac{\mu}{4} \| x_{k+1} -x^\star \|^2 .
$$
Combining the two above relations we get
\begin{eqnarray*}
&&\frac{1}{\sqrt{s}} (E_{k+1}  - E_k) 
+ \demi\sqrt{\mu} E_{k+1} + \left(\sqrt{\mu} \frac{\mu}{4} - 2\beta \mu  L\right) \| x_{k+1} -x^\star \|^2  + \\
&& \left(\frac{\sqrt{\mu}}{2s} + \frac{\mu }{\sqrt{s}} -L\left(             2\beta \mu + \frac{1 }{\sqrt{s}} + \frac{\sqrt{\mu}}{2}\right)\right) \|x_{k+1} - x_{k}\|^2 
+\frac{\beta}{2} \|\nabla f( x_{k+1})\|^2 \leq 0
\end{eqnarray*}
Let us examine the sign of the above quantities:
Under the condition $L \leq \frac{\sqrt{\mu}}{8 \beta}$ we have 
$\sqrt{\mu} \frac{\mu}{4} - 2\beta \mu  L \geq 0$.
Under the condition 
$L \leq \frac{\frac{\sqrt{\mu}}{2s} + \frac{\mu }{\sqrt{s}}}
{2\beta \mu + \frac{1 }{\sqrt{s}} + \frac{\sqrt{\mu}}{2}}$ we have
$\frac{\sqrt{\mu}}{2s} + \frac{\mu }{\sqrt{s}} -L\left(2\beta \mu + \frac{1 }{\sqrt{s}} + \frac{\sqrt{\mu}}{2}\right) \geq 0$.
Therefore, under the above conditions
$$
\frac{1}{\sqrt{s}} (E_{k+1}  - E_k) 
+ \demi\sqrt{\mu} E_{k+1}  +\frac{\beta}{2} \|\nabla f( x_{k+1})\|^2 \leq 0.
$$
Set $q= \frac{1}{1+ \demi \sqrt{\mu s}}$, which satisfies $0 <q <1$. By a similar argument as in Theorem \ref{prox-inertiel-strongly convex} 
$$
E_k \leq E_1 q^{k-1}.
$$
According to the definition of $E_k \geq  f(x_k) -f(x^\star)$, we finally obtain
$$
f(x_k) -f(x^\star) = \cO \left( q^{k} \right),
$$
and the linear convergence of $x_k$ to $x^\star$ and that of the gradients to zero. \qed
\end{proof}

\section{Numerical results}\label{sec:numerics}
Here, we illustrate our results on the composite problem on $\cH=\R^n$,
\begin{equation}\label{eq:minP}\tag{RLS}
\min_{x \in \R^n} \left\{ f(x) \eqdef \frac{1}{2}\norm{y-Ax}^2 + g(x) \right\} ,
\end{equation}
where $A$ is a linear operator from $\R^n$ to $\R^m$, $m \leq n$, $g: \R^n \to \Rb$ is a proper lsc convex function which acts as a regularizer. Problem~\eqref{eq:minP} is extremely popular in a variety of fields ranging from inverse problems in signal/image processing, to machine learning and statistics. Typical examples of $g$ include the $\ell_1$ norm (Lasso), the $\ell_1-\ell_2$ norm (group Lasso), the total variation, or the nuclear norm (the $\ell_1$ norm of the singular values of $x \in \R^{N \times N}$ identified with a vector in $\R^n$ with $n=N^2$). To avoid trivialities, we assume that the set of minimizers of~\eqref{eq:minP} is non-empty.

Though~\eqref{eq:minP} is a composite non-smooth problem, it fits perfectly well into our framework. Indeed, the key idea is to appropriately choose the metric. For a symmetric positive definite matrix $S \in \R^{n \times n}$, denote the scalar product in the metric $S$ as $\dotp{S\cdot}{\cdot}$ and the corresponding norm as $\norm{\cdot}_S$. When $S=I$, then we simply use the shorthand notation for the Euclidean scalar product $\dotp{\cdot}{\cdot}$ and norm $\norm{\cdot}$. For a proper convex lsc function $h$, we denote $h_S$ and $\prox_h^S$ its Moreau envelope and proximal mapping in the metric $S$, \ie
\[
h_S(x) = \min_{z \in \R^n} \frac{1}{2}\norm{z-x}_S^2 + h(z), \quad \prox^S_{h}(x) = \argmin_{z \in \R^n} \frac{1}{2}\norm{z-x}_S^2 + h(z) .
\] 
Similarly, when $S=I$, we drop $S$ in the above notation.

Let $M=s^{-1}I-A^*A$. With the proviso that $0 < s\norm{A}^2 < 1$, $M$ is a symmetric positive definite matrix. It can be easily shown (we provide a proof in Appendix~\ref{SS:proofproxFBM} for completeness; see also the discussion in~\cite[Section~4.6]{chambollereview}), that the proximal mapping of $f$ as defined in \eqref{eq:minP} in the metric $M$ is
\begin{equation}\label{eq:proxFBM}
\prox^M_{f}(x) = \prox_{s g}(x + s A^*(y-Ax)) ,
\end{equation}
which is nothing but the forward-backward fixed-point operator for the objective in \eqref{eq:minP}. Moreover, $f_M$ is a continuously differentiable convex function whose gradient (again in the metric $M$) is given by the standard identity
\[
\nabla f_M(x) = x - \prox^M_{f}(x) ,
\]
and $\norm{\nabla f_M(x)-\nabla f_M(z)}_M \leq \norm{x-z}_M$, \ie $\nabla f_M$ is Lipschitz continuous in the metric $M$. In addition, a standard argument shows that 
\[
\argmin_{\cH} f = \mathrm{Fix}(\prox^M_{f}) = \argmin_{\cH} f_M .
\]
We are then in position to solve \eqref{eq:minP} by simply applying \IGAHD (see Section~\ref{sec:igahd}) to $f_M$. We infer from Theorem~\ref{pr.decay_E_k} and properties of $f_M$ that
\[
f(\prox^M_f(x_k))-\min_{\R^n} f = \mathcal{O}(k^{-2}) .
\]
\IGAHD and FISTA (\ie \IGAHD with $\beta=0$) were applied to $f_M$ with four instances of $g$: $\ell_1$ norm, $\ell_1-\ell_2$ norm, the total variation, and the nuclear norm. The results are depicted in Figure~\ref{fig:rls}. One can clearly see that the convergence profiles observed for both algorithms agree with the predicted rate. Moreover, \IGAHD exhibits, as expected, less oscillations than FISTA, and eventually converges faster.

\begin{figure}
  \includegraphics[width=.5\textwidth]{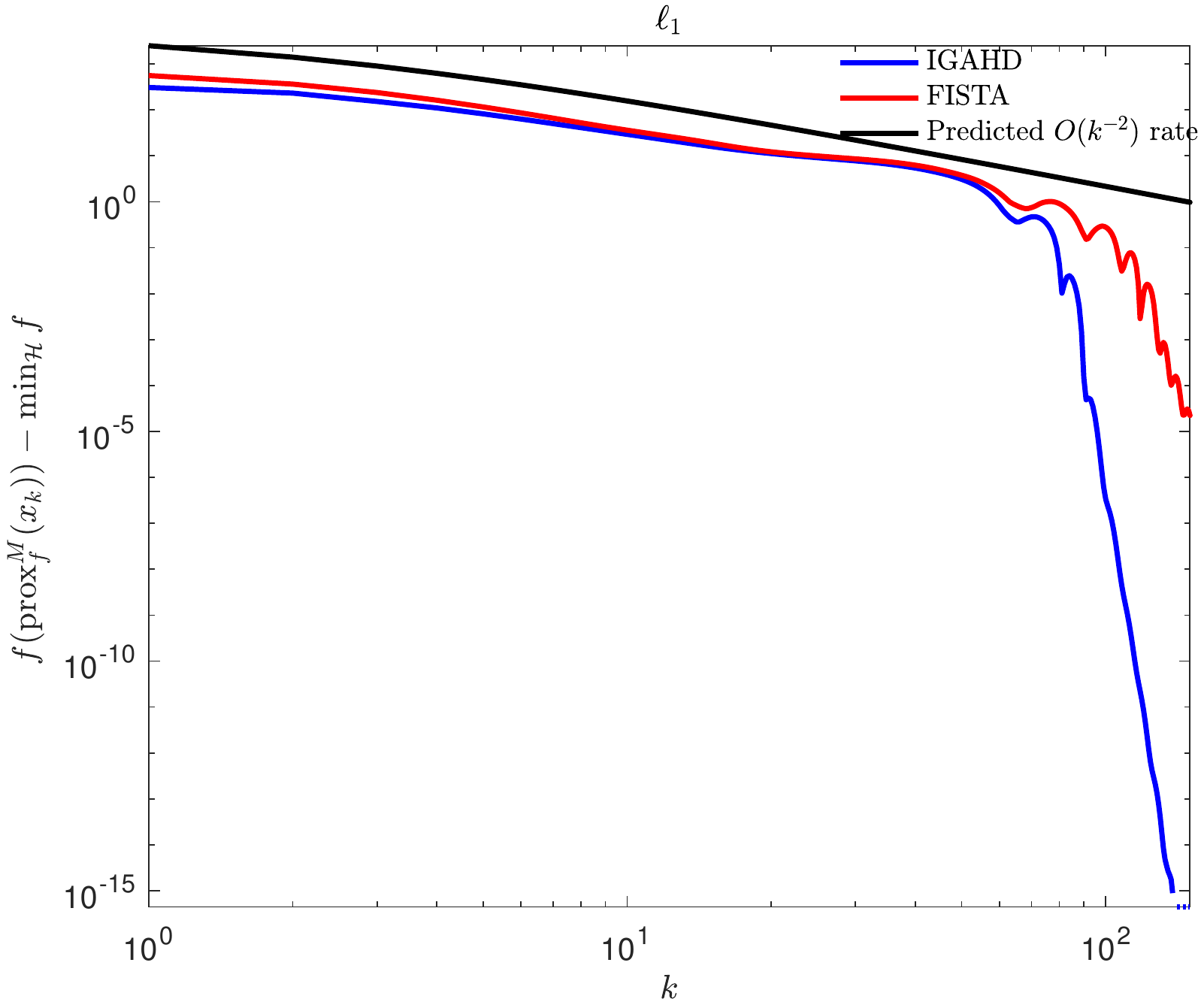}
  \includegraphics[width=.5\textwidth]{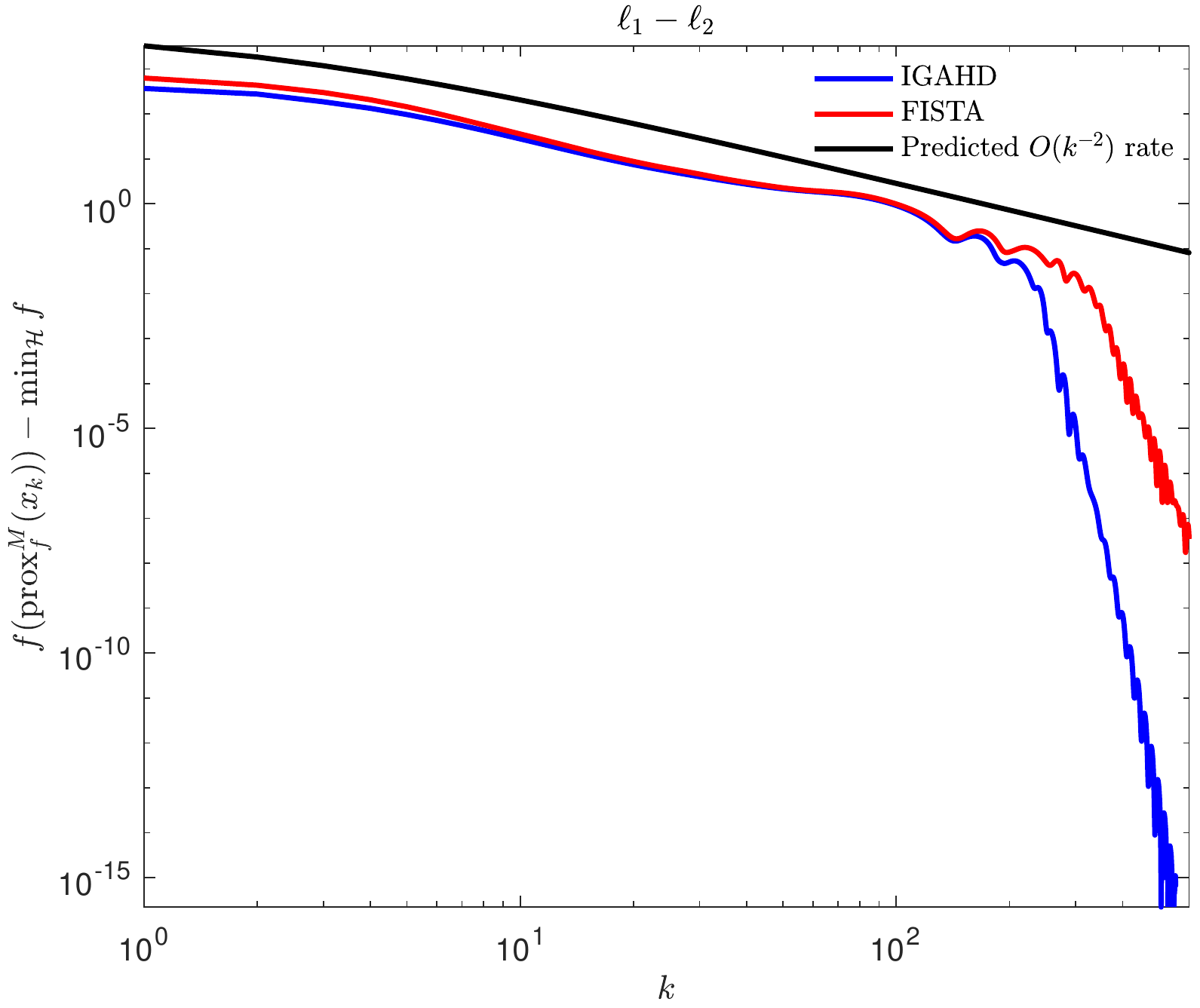}
  \includegraphics[width=.5\textwidth]{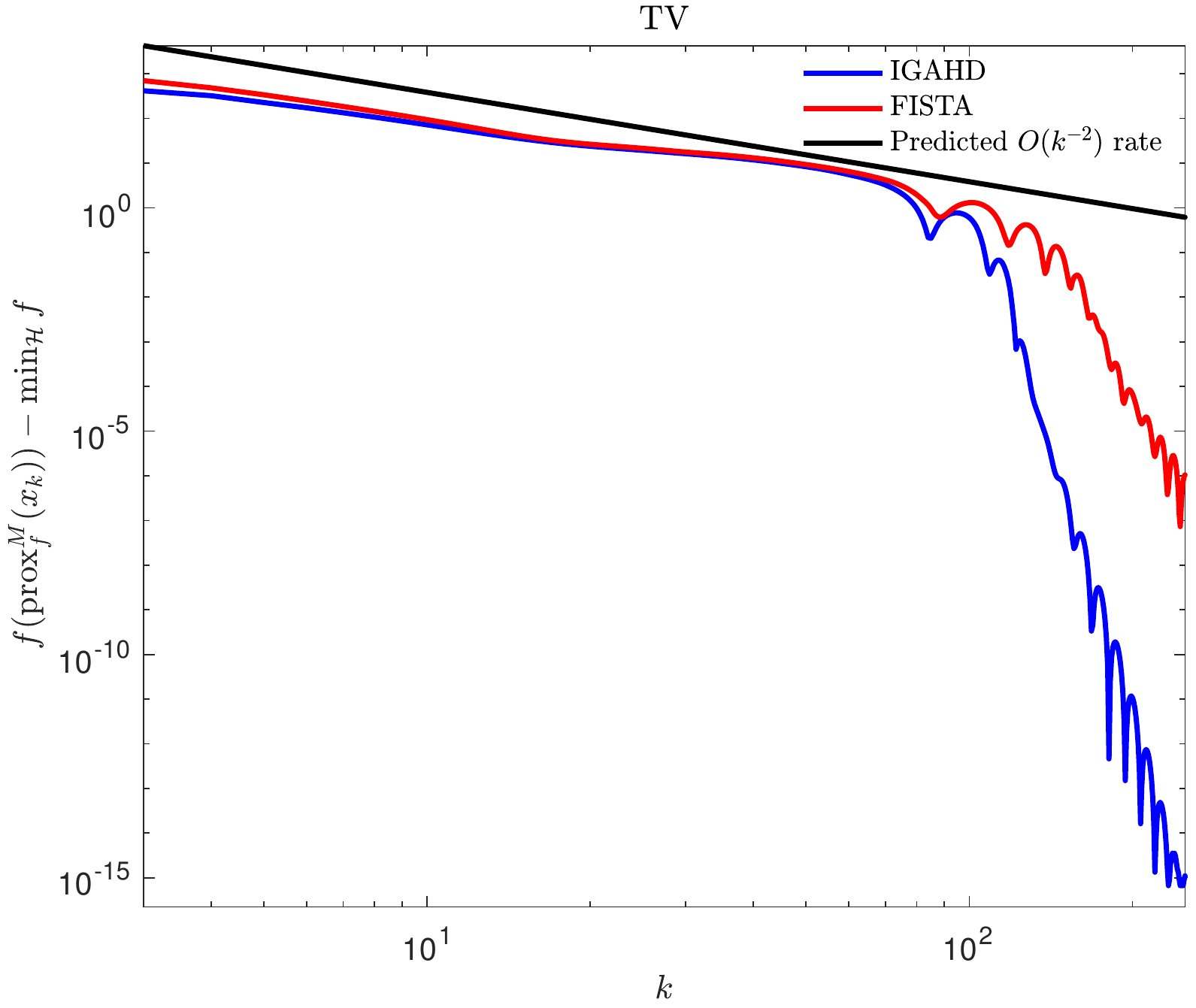}
  \includegraphics[width=.5\textwidth]{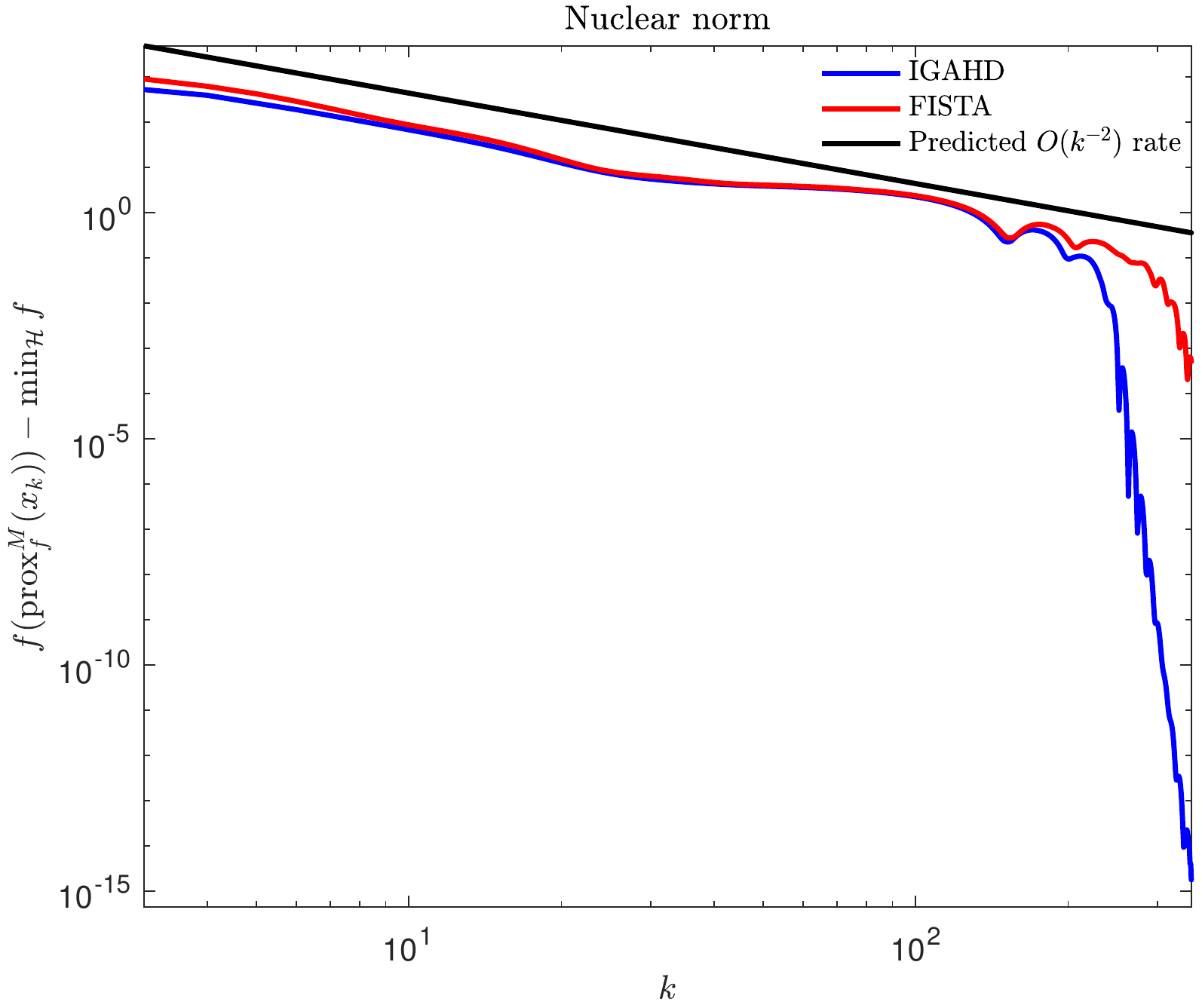}
\caption{Evolution of $f(\prox^M_f(x_k))-\min_{\R^n} f$, where $x_k$ is the iterate of either \IGAHD or FISTA, when solving \eqref{eq:minP} with different regularizers $g$.}
\label{fig:rls}
\end{figure}

\section{Conclusion, Perspectives}

As a guideline to our study, the inertial dynamics with Hessian driven damping give rise to a new class of first-order algorithms for convex optimization. While retaining the fast convergence of the function values reminiscent of the Nesterov accelerated algorithm, they benefit from  additional favorable properties among which the most important are:
\begin{enumerate}[label=$\bullet$]
\item fast convergence of  gradients towards zero;
\item global convergence of the iterates to optimal solutions;
\item extension to the non-smooth setting;
\item acceleration via time scaling factors.
\end{enumerate}
This article contains the core of our study with a particular focus on the gradient and proximal methods. The results thus obtained pave the way to new research avenues. For instance:
\begin{enumerate}[label=$\bullet$]
\item as initiated in Section~\ref{sec:numerics}, apply these results to structured composite optimization problems beyond \eqref{eq:minP} and develop corresponding splitting algorithms; 
 \item with the additional gradient estimates, we can expect the restart method to work better with the presence of the Hessian damping term;
\item deepen the link between our study and the Newton and Levenberg-Marquardt dynamics and algorithms (e.g., \cite{ASv}), and with the Ravine method~\cite{GZ}.
\item the inertial dynamic with Hessian driven damping goes well with tame analysis and Kurdyka-Lojasiewicz property \cite{AABR}, suggesting that the corresponding algorithms be developed in a non-convex (or even non-smooth) setting.
\end{enumerate}

\appendix
\section{Auxiliary results}

\subsection{Extended descent lemma} \label{SS:extended descent lemma}
\begin{lemma}\label{ext_descent_lemma}
Let  $f: \cH \to \R$ be  a  convex function whose gradient is $L$-Lipschitz continuous. Let $s \in ]0,1/L]$. Then for all $(x,y) \in \cH^2$, we have
\begin{equation}\label{eq:extdesclem}
f(y - s \nabla f (y)) \leq f (x) + \left\langle  \nabla f (y), y-x \right\rangle -\frac{s}{2} \|  \nabla f (y) \|^2 -\frac{s}{2} \| \nabla f (x)- \nabla f (y) \|^2 .
\end{equation}
\end{lemma}


\begin{proof}
Denote $y^+=y - s \nabla f (y)$. By the standard descent lemma applied to $y^+$ and $y$, and since $sL \leq 1$ we have
\begin{equation}\label{eq:descfm2}
f(y^+) \leq f(y) - \frac{s}{2}\pa{2-Ls} \| \nabla f (y) \|^2 \leq f(y) - \frac{s}{2} \|  \nabla f (y) \|^2.
\end{equation}
We now argue by duality between strong convexity and Lipschitz continuity of the gradient of a convex function. Indeed, using Fenchel identity, we have
\[
f(y) = \dotp{\nabla f(y)}{y} - f^*(\nabla f(y)) .
\]
$L$-Lipschitz continuity of the gradient of $f$ is equivalent to $1/L$-strong convexity of its conjugate $f^*$. This together with the fact that $(\nabla f)^{-1}=\partial f^*$ gives for all $(x,y) \in \cH^2$,
\[
f^*(\nabla f(y)) \geq  f^*(\nabla f(x)) + \dotp{x}{\nabla f(y)-\nabla f(x)} + \frac{1}{2L}\norm{\nabla f(x)-\nabla f(y)}^2 .
\]
Inserting this inequality into the Fenchel identity above yields
\begin{align*}
f(y) 
&\leq - f^*(\nabla f(x)) + \dotp{\nabla f(y)}{y} - \dotp{x}{\nabla f(y)-\nabla f(x)} - \frac{1}{2L}\norm{\nabla f(x)-\nabla f(y)}^2 \\
&= - f^*(\nabla f(x)) + \dotp{x}{\nabla f(x)} + \dotp{\nabla f(y)}{y-x} - \frac{1}{2L}\norm{\nabla f(x)-\nabla f(y)}^2 \\
&= f(x) + \dotp{\nabla f(y)}{y-x} - \frac{1}{2L}\norm{\nabla f(x)-\nabla f(y)}^2  \\
&\leq f(x) + \dotp{\nabla f(y)}{y-x} - \frac{s}{2}\norm{\nabla f(x)-\nabla f(y)}^2 .
\end{align*}
Inserting the last bound into \eqref{eq:descfm2} completes the proof.
\end{proof}

\subsection{Proof of \eqref{eq:proxFBM}} \label{SS:proofproxFBM}
\begin{proof}We have
\begin{align*}
\prox_{f}^{M}(x) 
&= \argmin_{z \in \R^n} \frac{1}{2}\norm{z - x}_M^2 + f(z) \\
&= \argmin_{z \in \R^n} \frac{1}{2s}\norm{z - x}^2 - \frac{1}{2}\norm{A(z - x)}^2 + \frac{1}{2}\norm{y-A z}^2 + g(z) .
\end{align*}
By the Pythagoras relation, we then get
\begin{align*}
\prox_{f}^M(x) &= \argmin_{z \in \R^n} \frac{1}{2s}\norm{z - x}^2 + \frac{1}{2}\norm{y-A x}^2 - \dotp{A(x-z)}{A x - y} + g(z) \\
				      &= \argmin_{z \in \R^n} \frac{1}{2s}\norm{z - x}^2 - \dotp{z - x}{A^*\pa{y - A x}} + g(z) \\
				      &= \argmin_{z \in \R^n} \frac{1}{2s}\norm{z - \pa{x - s A^*\pa{A x - y}}}^2 + g(z) \\
				      &= \prox_{s g}\pa{x - s A^*\pa{A x - y}} .
\end{align*}
\end{proof}

\subsection{Closed-form solutions of \DINAVD{\alpha,\beta,b} for quadratic functions} \label{SS:solquad}
We here provide the closed form solutions to \DINAVD{\alpha,\beta,b} for the quadratic objective $f: \R^n \to \dotp{Ax}{x}$, where $A$ is a symmetric positive definite matrix. The case of a semidefinite positive matrix $A$ can be treated similarly by restricting the analysis to $\ker(A)^\top$. Projecting \DINAVD{\alpha,\beta,b} on the eigenspace of $A$, one has to solve $n$ independent one-dimensional ODEs of the form
\begin{equation*}
\ddot{x}_i(t) + \pa{\frac{\alpha}{t}+\beta(t)\lambda_i}\dot{x}_i(t) + \lambda_i b(t) x_i(t) = 0, \qquad i=1,\ldots,n .
\end{equation*}
where $\lambda_i > 0$ is an eigenvalue of $A$. In the following, we drop the subscript $i$.

\paragraph{\textbf{Case~$\boldsymbol{\beta(t) \equiv \beta, b(t)=b+\gamma/t, \beta \geq 0, b > 0, \gamma \geq 0}$:}}
The ODE reads
\begin{equation}\label{eq:dinavdquad1D}
\ddot{x}(t) + \pa{\frac{\alpha}{t}+\beta\lambda}\dot{x}(t) + \lambda \pa{b+\frac{\gamma}{t}} x(t) = 0 .
\end{equation}
\begin{enumerate}[label=$\bullet$]
\item If $\beta^2\lambda^2 \neq 4b\lambda$: set
\[
\xi = \sqrt{\beta^2\lambda^2 - 4b\lambda}, \, \kappa = \lambda\frac{\gamma-\alpha\beta/2}{\xi}, \, \sigma = (\alpha-1)/2 .
\]
Using the relationship between the Whitaker functions and the Kummer's confluent hypergeometric functions $M$ and $U$, see~\cite{Bateman}, the solution to \eqref{eq:dinavdquad1D} can be shown to take the form
\[
x(t) = \xi^{\alpha/2} e^{-(\beta\lambda+\xi)t/2}\brac{c_1 M(\alpha/2-\kappa,\alpha,\xi t) + c_2 U(\alpha/2-\kappa,\alpha,\xi t)} ,
\]
where $c_1$ and $c_2$ are constants given by the initial conditions.

\item If $\beta^2\lambda^2 = 4b\lambda$: set $\zeta=2\sqrt{\lambda\pa{\gamma-\alpha\beta/2}}$. The solution to \eqref{eq:dinavdquad1D} takes the form
\[
x(t) = t^{-\pa{\alpha-1}/2}e^{-\beta\lambda t/2}\brac{c_1 J_{(\alpha-1)/2}(\zeta \sqrt{t}) + c_2 Y_{(\alpha-1)/2}(\zeta \sqrt{t})} ,
\]
where $J_\nu$ and $Y_\nu$ are the Bessel functions of the first and second kind.
\end{enumerate}
When $\beta > 0$, one can clearly see the exponential decrease forced by the Hessian. From the asymptotic expansions of $M$, $U$, $J_{\nu}$ and $Y_{\nu}$ for large $t$, straightforward computations provide the behaviour of $|x(t)|$ for large $t$ as follows:
\begin{enumerate}[label=$\bullet$]
\item If $\beta^2\lambda^2 > 4b\lambda$, we have
\[
|x(t)| = \cO\pa{t^{-\frac{\alpha}{2}+|\kappa|} e^{-\frac{\beta\lambda-\xi}{2}t}} = \cO\pa{e^{-\frac{2b}{\beta}t - \pa{\frac{\alpha}{2}-|\kappa|}\log(t)}} .
\]

\item If $\beta^2\lambda^2 < 4b\lambda$, whence $\xi \in i \R^+_*$ and $\kappa \in i \R$, we have
\[
|x(t)| = \cO\pa{t^{-\frac{\alpha}{2}} e^{-\frac{\beta\lambda}{2}t}} .
\]
\item If $\beta^2\lambda^2 = 4b\lambda$, we have
\[
|x(t)| = \cO\pa{t^{-\frac{2\alpha-1}{4}} e^{-\frac{\beta\lambda}{2}t}} .
\]
\end{enumerate}

\paragraph{\textbf{Case~$\boldsymbol{\beta(t) = t^{\beta}, b(t)=ct^{\beta-1}, \beta \geq 0, c > 0}$:}}
The ODE reads now
\begin{equation*}
\ddot{x}(t) + \pa{\frac{\alpha}{t}+t^\beta\lambda}\dot{x}(t) + c\lambda t^{\beta-1} x(t) = 0 .
\end{equation*}
Let us make the change of variable $t \eqdef \tau^{\frac{1}{\beta+1}}$. Let $y(\tau) \eqdef x\pa{\tau^{\frac{1}{\beta+1}}}$. By the standard derivation chain rule, it is straightforward to show that $y$ obeys the ODE
\begin{equation*}
\ddot{y}(\tau) + \pa{\frac{\alpha+\beta}{(1+\beta)\tau}+\frac{\lambda}{1+\beta}}\dot{y}(\tau) + \frac{c\lambda}{(1+\beta)^2\tau} y(\tau) = 0 .
\end{equation*}
It is clear that this is a special case of~\eqref{eq:dinavdquad1D}. Since $\beta$ and $\lambda > 0$, set
\[
\xi = \frac{\lambda}{1+\beta}, \, \kappa = -\frac{\alpha+\beta-c}{1+\beta}, \, \sigma = \frac{\alpha+\beta}{2(1+\beta)} - \frac{1}{2} .
\]
It follows from the first case above that
\[
x(t) = \xi^{\sigma+1/2} e^{-\frac{\lambda\tau}{1+\beta}}\brac{c_1 M\pa{\sigma-\kappa+1/2,\frac{\alpha+\beta}{1+\beta},\xi \tau} + c_2 U\pa{\sigma-\kappa+1/2,\frac{\alpha+\beta}{1+\beta},\xi \tau}} .
\]
Asymptotic estimates can also be derived similarly to above. We omit the details for the sake of brevity.


\begin{thebibliography}{10}

\bibitem{Alv0}{\sc F. \'Alvarez},  {\it  On the minimizing property of a second-order dissipative system in Hilbert spaces}, SIAM J. Control Optim., 38 (2000), No. 4, pp.~1102-1119.




\bibitem{AABR}{\sc F. \'Alvarez, H. Attouch, J. Bolte, P. Redont}, {\it A second-order gradient-like dissipative dynamical system with Hessian-driven damping. Application to optimization and mechanics}, J. Math. Pures Appl.,  81 (2002), No. 8, pp.~ 747--779.

\bibitem{AAD}\tcb{{\sc V. Apidopoulos, J.-F. Aujol,  Ch. Dossal}, {\it Convergence rate of inertial Forward-Backward algorithm beyond Nesterov's rule}, Math. Program. Ser. B., 180 (2020), pp.~137-?156.}
 


\bibitem{AC1} {\sc H. Attouch, A.  Cabot}, {\it Asymptotic stabilization of inertial gradient dynamics with time-dependent viscosity},  J. Differential Equations, 263 (2017), pp.~5412-5458.
 
\bibitem{AC2} {\sc H. Attouch, A.  Cabot}, {\it Convergence rates of inertial forward-backward algorithms},  SIAM J. Optim., 28 (1) (2018), pp.~849--874.





 
 
 
\bibitem{AC2R-EECT} {\sc H. Attouch, A.  Cabot, Z. Chbani, H. Riahi}, { \it Rate of convergence of inertial gradient dynamics with time-dependent viscous damping coefficient}, Evolution Equations and Control Theory,  7 (2018), No. 3, pp.~353--371.

\bibitem{ACR-rescale} \tcb{{\sc H. Attouch, Z. Chbani, H. Riahi}, {\it Fast proximal methods via time scaling of  damped inertial dynamics}, SIAM J. Optim., 29 (2019), No. 3, pp.~2227?-2256.}
 
\bibitem{ACPR} {\sc H. Attouch,  Z. Chbani, J. Peypouquet, P. Redont},  {\it Fast convergence of inertial dynamics and algorithms with asymptotic vanishing viscosity}, Math. Program. Ser. B., 168 (2018), pp.~123--175.
  

\bibitem{ACR-subcrit} {\sc H. Attouch,  Z. Chbani, H. Riahi},
{\it Rate of convergence  of the Nesterov accelerated gradient method  in the subcritical case  $\alpha \leq 3$}, ESAIM Control Optim. Calc. Var., 25 (2019), pp.~2-35.
  

  
\bibitem{AP} {\sc H. Attouch, J. Peypouquet}, {\it The rate of convergence of Nesterov's accelerated forward-backward method is actually faster than $1/k^2$}, SIAM J. Optim., 26 (2016), No. 3, pp.~1824--1834.


\bibitem{APR} {\sc H. Attouch, J. Peypouquet, P. Redont},  {\it  A dynamical approach to an inertial forward-backward algorithm for convex minimization}, SIAM J. Optim., 24  (2014), No. 1, pp.~232--256. 

\bibitem{APR2} {\sc H. Attouch, J. Peypouquet, P. Redont},  {\it  Fast convex minimization via inertial dynamics with Hessian driven damping},
J. Differential Equations, 261, No. 10, (2016), pp.~5734--5783. 

\bibitem{ASv} {\sc H. Attouch, B. F. Svaiter},
 {\it A continuous dynamical Newton-Like approach to solving monotone inclusions},   SIAM J. Control Optim., 49  (2011), No. 2,  pp.~574--598.

{\it Global convergence of a closed-loop regularized Newton method for solving monotone inclusions in Hilbert spaces},  J.  Optim. Theory  Appl., 157 (2013), No. 3, pp.~624--650. 

\bibitem{AD}{\sc J.-F. Aujol, Ch. Dossal}, {\it Stability of over-relaxations for the Forward-Backward
algorithm, application to FISTA},  SIAM J. Optim., 25 (2015), No. 4, pp.~2408--2433.

\bibitem{AD17}{\sc J.-F. Aujol, Ch. Dossal}, {\it Optimal rate of convergence 
of an ODE associated to the Fast Gradient Descent schemes for} $b>0$, 2017,  https://hal.inria.fr/hal-01547251v2. 


\bibitem{Bateman}{\sc H. Bateman}, {\it  Higher transcendental functions}, McGraw-Hill, Vol. 1, (1953).

\bibitem{BC}{\sc H. Bauschke, P. L. Combettes}, {\it  Convex Analysis and Monotone Operator Theory in Hilbert spaces}, CMS Books in Mathematics, Springer,   (2011).

\bibitem{BT}{\sc A. Beck, M. Teboulle},  {\it  A fast iterative shrinkage-thresholding algorithm for linear inverse problems},  SIAM J. Imaging Sci., 2  (2009),  No. 1, pp.~183--202.







\bibitem{Bre1}{\sc H. Br\'ezis}, {\it  Op\'erateurs maximaux monotones dans les espaces de Hilbert et \'equations d'\'evolution}, Lecture Notes 5, North Holland, (1972).

 in Machine Learning,   8(3-4)  (2015), pp.~231--358. 
 
\bibitem{CEG1}{\sc  A. Cabot, H. Engler, S. Gadat}, {\it  On the long time behavior of second order differential equations with asymptotically small dissipation}, Transactions of the American Mathematical Society, 361 (2009), pp.~5983--6017.


\bibitem{CD}{\sc  A. Chambolle, Ch. Dossal}, {\it  On the convergence of the iterates of the Fast Iterative Shrinkage Thresholding Algorithm}, Journal of Optimization Theory and Applications, 166 (2015), pp.~968--982.

\bibitem{chambollereview}{\sc  A. Chambolle, T. Pock}, {\it An introduction to continuous optimization for imaging}, Acta Numerica, 25 (2016), pp.~161--319.




\bibitem{GZ}{I.M. Gelfand, M. Zejtlin}, {\it Printszip nelokalnogo poiska v sistemah avtomatich}, Optimizatsii, Dokl. AN SSSR, 137 (1961), pp.~295?298 (in Russian).

%
%
















\bibitem{May}{\sc R. May}, {\it Asymptotic for a second-order evolution equation with convex potential and vanishing damping term}, Turkish Journal of Math., 41(3) (2017), pp.~681--685.


\bibitem{Nest1}{\sc  Y. Nesterov}, {\it A method of solving a convex programming problem with convergence rate
$O(1/k^2)$}, Soviet Mathematics Doklady, 27 (1983), pp.~ 372--376.



\bibitem{Nest4}{\sc  Y. Nesterov}, {\it  Gradient methods for minimizing composite objective function}, Math. Program., Volume 152(1-2) (2015), pp.~381--404 






\bibitem{Pol} {\sc B.T. Polyak}, {\it Some methods of speeding up the convergence of iteration methods}, U.S.S.R. Comput. Math. Math. Phys., 4 (1964), pp.~1--17.

\bibitem{Polyak2}{\sc B.T. Polyak}, {\it  Introduction to optimization}. New York: Optimization Software. (1987).



\bibitem{Siegel}{\sc W. Siegel}, {\it Accelerated first-order methods: Differential equations and Lyapunov functions}, arXiv:1903.05671v1  [math.OC], 2019.


\bibitem{SDJS}{\sc  B. Shi, S.  S. Du,  M. I. Jordan,  W. J. Su}, {\it Understanding the acceleration phenomenon via high-resolution differential equations}, arXiv:submit/2440124[cs.LG] 21 Oct 2018.

\bibitem{SBC}{\sc W. J. Su,  S. Boyd,  E. J. Cand\`es}, {\it A differential equation for modeling Nesterov's accelerated gradient method: theory and insights}. NIPS'14, 27 (2014), pp.~2510--2518. 

\bibitem{WRJ}{\sc A. C. Wilson, B. Recht,  M. I. Jordan}, {\it A Lyapunov analysis of momentum methods in optimization}. arXiv preprint arXiv:1611.02635, 2016.


\end{thebibliography}
\end{document}